\documentclass[10pt]{article}
\usepackage{color}
\usepackage{amssymb}
\usepackage{amsthm,array,amssymb,amscd,amsfonts,latexsym, url}
\usepackage{amsmath}

\newtheorem{theo}{Theorem}[section]
\newtheorem{prop}[theo]{Proposition}

\newtheorem{lemm}[theo]{Lemma}
\newtheorem{coro}[theo]{Corollary}
\newtheorem{rema}[theo]{Remark}
\newtheorem{Defi}[theo]{Definition}

\newtheorem{conj}[theo]{Conjecture}

\newtheorem{question}[theo]{Question}
\newtheorem{example}[theo]{Example}
\newtheorem{sublemm}[theo]{Sublemma}
\title{Torsion points of sections of Lagrangian torus fibrations and the Chow ring of hyper-K\"ahler manifolds}
\author{Claire Voisin\footnote{Coll\`ege de France and ETH-ITS}}

\date{}

\newfont{\gothic}{eufb10}
\begin{document}
\maketitle
\setcounter{section}{-1}

\begin{abstract} Let $\phi:X\rightarrow B$ be  a Lagrangian
fibration  on a projective irreducible hyper-K\"ahler manifold. Let $M\in {\rm Pic}\,X$ be a line bundle whose restriction to the general fiber $X_b$ of
 $\phi$ is topologically trivial. We prove that if the fibration is isotrivial or has maximal variation  and $X$ is of dimension $\leq 8$, the set of points $b$ such that the restriction
 $M_{\mid X_b}$ is torsion is dense in $B$. We give an application to the Chow ring of $X$, providing further evidence for Beauville's weak splitting conjecture.
\end{abstract}
\section{Introduction}
\subsection{The general problem and main result}
Let $X$ be an  algebraic variety over $\mathbb{C}$,
 $X\rightarrow B$  a smooth projective morphism and let $M\in {\rm Pic}\,X$ be a line bundle
which is topologically trivial on the fibers of $\phi$. The line bundle $M$ then determines an algebraic section $\nu_M$
of the torus
fibration ${\rm Pic}^0(X^0/B^0)\rightarrow B^0$ (which we will call the normal function
associated to $M$),  where
$X^0\rightarrow B^0$ is the restriction of $\phi $ over the open set $B^0\subset B$ of  regular values of $\phi$. The group of such sections is called the Mordell-Weil group of the fibration.
The problem we investigate
in this paper is:
\begin{question}
\label{question}
When do there exist points $b\in B^0$ where $\nu_M(b)$ is a torsion point? Are these points dense
(for the usual or Zariski topology) in $B$?
\end{question}
Note that this question can be restated in terms of variations of mixed Hodge structures $E_t$
obtained as an extension
$$0\rightarrow H_t\rightarrow E_t\rightarrow \mathbb{Z}\rightarrow 0$$
of $\mathbb{Z}$ by a pure Hodge structure $H_t$ of weight $1$ (see for example \cite{carlson}, \cite{brosnan},  \cite{charles}).
Torsion points then correspond to those points $t$ for which  $E_t$ has a nonzero rational element which belongs to $F^1E_{t,\mathbb{C}}$.
In this  setting, similar existence or density problems have been investigated
in the study of the density of the Noether-Lefschetz locus (see  \cite{ciliharrismiranda}, \cite{voisindensity}), where variations of weight $2$ Hodge structures are considered and special points $t$ are those for which the Hodge structure $E_t$ acquires a Hodge class. More generally, the set of special points is the object of a vast literature related to Andr\'{e}-Oort conjecture  and
``unlikely intersection'' theory \cite{zannierbook}, but more focused on bounding special points than on their existence. It turns out that both problems are related to transversality questions so that
 there is in fact  a certain overlap in the methods (see for example \cite{zannier}).

Question \ref{question} is natural only when the codimension of the locus of torsion points in
${\rm Pic}^0(X/B)$, that is $h^{1,0}(X_b)={\rm dim}\,{\rm Pic}^0(X/B)-{\rm dim}\,B$, is not greater
than $n={\rm dim}\,B$, so that the section $\nu_M$ is expected to meet this locus.
We will consider in this paper the case where $h^{1,0}(X_b)={\rm dim}\,B$, and more precisely
we will work in the following setting:
 $X$ will be a projective hyper-K\"ahler manifold of dimension $2n$ and  $\phi:X\rightarrow B$ will be a Lagrangian fibration.
According to Matsushita \cite{matsumain}, any morphism $\phi: X\rightarrow B$ with $0<{\rm dim}\,B<{\rm dim}\,X$
and connected fibers provides such  a Lagrangian fibration. The smooth fibers $X_b$ are then abelian varieties of dimension $n$, which are in  fact canonically polarized, by a  result of Matsushita
 (see \cite[Lemma 2.2]{matsurank}, see also Proposition \ref{promatsu}). So in the following definition,
 the local period map could be replaced in our case  by a moduli map from the base to a moduli space of abelian varieties.

 \begin{Defi} We will say that a smooth torus fibration $X^0\rightarrow B^0$
 has maximal variation
 if the associated (locally defined) period map $B^0\rightarrow \mathcal{D}$ is  generically of maximal
rank $n={\rm dim}\,B$.
\end{Defi}
 The fibration is said to satisfy Matsushita's alternative if it has maximal
  variation or it is   locally constant (the isotrivial case)
Matsushita conjectured that   this alternative holds for Lagrangian fibrations on projective
hyper-K\"ahler manifolds.
This conjecture was proved in \cite{geemenvoisin} assuming
that the Mumford-Tate group of the Hodge structure on the transcendental
 cohomology $H^2(X,\mathbb{Q})_{tr}\subset H^2(X,\mathbb{Q})$ is the full special orthogonal group
 of $H^2(X,\mathbb{Q})_{tr}$ equipped with the Beauville-Bogomolov intersection form, and that $b_2(X)_{tr}\geq 5$.
In particular, it is satisfied by general deformations of $X$ with fixed Picard lattice, assuming
$b_2(X)_{tr}\geq 5$.

Our main result in this paper is the following:
\begin{theo}\label{theomain} Let $X,\,\phi$ be as above, and
let $M\in {\rm Pic}\,X$ restrict to a topologically trivial line bundle
on the smooth fibers $X_b$ of $\phi$. Assume that either the
 fibration is locally isotrivial or ${\rm dim}\,X\leq 8$ and
the variation is maximal. Then the set of points
$b\in B^0$ such that $\nu_M(b)=M_{\mid X_b}$ is a torsion point in ${\rm Pic}\,X_b$ is dense in $B^0$ for the usual topology.
\end{theo}

Applying the result  of \cite{geemenvoisin} mentioned above, we conclude:
\begin{coro} \label{corogen}  Let $X,\,\phi,\,M$ be as in Theorem \ref{theomain}. Assume  $b_2(X)\geq8$. Then the very general deformation $X_t$ of $X$ preserving $L$, $M$  and an ample line bundle on  $X$
satisfies the conclusion of Theorem \ref{theomain}.
\end{coro}
The restriction to dimension $\leq8$ is certainly not essential here and we believe that
 the result is true in any dimension although the
  analysis seems to be very complicated in higher dimension. The restriction to dimension $8$ appears in the analysis of the constraints on the
infinitesimal variation of Hodge structures imposed by contradicting the conclusion of Theorem
\ref{theomain}. In dimension $10$ we completely  classify the possible (although improbable) situation  where the conclusion of Theorem \ref{theomain} does not hold (see Proposition \ref{prodim10} proved  in Section \ref{secn5}).

In order to introduce
some of the  ingredients used in the proof  of  Theorem
\ref{theomain},
we will start with  the general case of an  elliptic fibration, (and without any hyper-K\"ahler condition). One gets in this case the following result, which, although we could not find a reference,
 is presumably not new, but will serve
as  a toy example for the strategy used.
\begin{theo}\label{theoelliptic} Let $\phi:X\rightarrow B$ be an elliptic fibration, where $B$ is a smooth
projective variety and $X$ is smooth projective, and let $B^0\subset B$ be the open set of regular values of
$\phi$. Let $M$ be a line bundle on $X$ which is of degree $0$
on the fibers of $\phi$. Then either  the set of points
$b\in B^0$ such that $\sigma_M(b)=M_{\mid X_b}$ is a torsion point in ${\rm Pic}\,X_b$ is dense in $B^0$ for the usual topology, or the restriction map $H^1(X,\mathbb{Q})\rightarrow H^1(X_b,\mathbb{Q})$ is surjective. In the second case, the fibration $\phi$ is locally isotrivial and the associated Jacobian fibration is rationally isogenous
 over $B$ to the product $J(X_b)\times B$.

\end{theo}

 Our approach to Theorem \ref{theomain} is infinitesimal and completed by an analysis of the monodromy. It uses the easy Proposition
\ref{ledense} which says that the torsion points of $\nu$ are dense in $B$ if the
natural locally defined map   $f_\nu:B\rightarrow H_1(A_{t_0},\mathbb{R})$ obtained
from $\nu$  by a
real analytic trivialization of the family of complex tori $A\rightarrow B$ is generically
submersive. This map is called the Betti map in \cite{ACZ}, \cite{zannier}.
A key tool for  our work  is the following very useful result of Andr\'e-Corvaja-Zannier
\cite{ACZ}.
\begin{theo}\label{theoandre}  Let $\pi: A\rightarrow B$ be a family of abelian varieties of dimension $d$, where $B$ is quasi-projective over $\mathbb{C}$ and ${\rm dim}\,B\geq d$.
Let $\nu$ be an algebraic section of $\pi$. Assume

 (i) The family has no locally trivial subfamily.

  (ii) The multiples $m \nu(B)$ are Zariski dense in $A$.

  Then if the real analytic map $f_\nu$ is never of maximal rank, the variation of Hodge structure
  of weight $1$ is degenerate in the following sense:

  For any $b\in B$, for any $\lambda\in H^{1,0}(A_b)$ the map
  $$\overline{\nabla}_\lambda:T_{B,b}\rightarrow H^{0,1}(A_b)$$
  is not of maximal rank.

\end{theo}

We refer to  Section
\ref{secACZ}  for the definition of $\overline{\nabla}$. We will also for completeness sketch the proof of this result in that section.

\subsection{An application to the Chow ring of hyper-K\"ahler fourfolds}
Let us explain one consequence of Theorem \ref{theomain}, which was our motivation for this work.
Following the discovery made in \cite{beauvoi} that a projective $K3$ surface $S$ has
a canonical $0$-cycle  $o_S$ with the property that for any two divisors $D,\,D'$
on $S$, the intersection $D\cdot D'$ is a multiple of $o_S$ in ${\rm CH}_0(S)$, Beauville conjectured the
following:
\begin{conj} \label{conjbeau} (See \cite{beau}). Let $X$ be a projective hyper-K\"ahler manifold. Then the cohomological cycle
class restricted to the subalgebra of ${\rm CH}(X)_\mathbb{Q}$ generated by divisor classes
is injective.
\end{conj}

This conjecture has been proved in \cite{voisinpamq} for varieties of the form
$S^{[n]}$, where $S$ is a
$K3$ surface and $n\leq 2b_2(S)_{tr}+4$, and in \cite{fulie} for generalized Kummer varieties. It is also proved in \cite{voisinpamq} for Fano varieties of lines
of cubic fourfolds, which are well-known to be irreducible hyper-K\"ahler fourfolds of $K3^{[2]}$ deformation type
(see \cite{beaudo}).
Finally, Conjecture \ref{conjbeau} is proved by Riess \cite{riess2} in the case of irreducible hyper-K\"ahler varieties of $K3^{[n]}$ or generalized Kummer deformation type admitting a Lagrangian fibration. Here the condition that
$X$ is a deformation of a punctual Hilbert scheme of a $K3$ surface guarantees, according to \cite{matsushita13},
that $X$ satisfies the   conjecture  that any nef line bundle $L$  on $X$ which is isotropic
for the Beauville-Bogomolov quadratic form is  the pull-back of a $\mathbb{Q}$-line bundle
on the basis $B$ of a Lagrangian fibration on $X$.

The application to Conjecture \ref{conjbeau}  we give in this paper  concerns one natural relation in the cohomology ring
of a hyper-K\"ahler manifold equiped with a Lagrangian fibration.
Note that, according to \cite{bover}, the set of polynomial relations between degree $2$ cohomology classes
on an irreducible hyper-K\"ahler $2n$-fold $X$ is generated as an ideal of ${\rm Sym}\,H^2(X,\mathbb{R})$ by
the relations
$$ \alpha^{n+1}=0\,\,{\rm in}\,\,H^{2n+2}(X,\mathbb{R})\,\,{\rm if}\,\,q(\alpha)=0.$$
For the same reason, the cohomological relations between divisor classes $d\in {\rm NS}(X)_{\mathbb{R}}\subset H^2(X,\mathbb{R})$
are generated by the relations
\begin{eqnarray}\label{eqrelcoh}d^{n+1}=0\,\,{\rm in}\,\,H^{2n+2}(X,\mathbb{R})\,\,\,\,{\rm if}\,\,q(d)=0,
\end{eqnarray}
In particular there are no relations in degree $\leq 2n$.
If $X$ is any projective hyper-K\"ahler manifold,  let $l\in H^2(X,\mathbb{Q})$ be such  that $q(l)=0$. One  has
then $l^{n+1}=0$ in
$H^{2n+2}(X,\mathbb{Q})$. If furthermore
$X$
 admits a Lagrangian fibration such that $l=c_1(L)$ and $L$ is pulled-back
 from the base, then one clearly has $L^{n+1}=0$ in ${\rm CH}(X)$ and this is the starting point of Riess' work
 \cite{riess1}, \cite{riess2}.
Observe next that by differentiation of
(\ref{eqrelcoh}) along the tangent space to $Q$ at $l$, namely
$l^{\perp q}$, one gets the relations
\begin{eqnarray}\label{eqrelcoh2}l^{n}h=0\,\,\,\,{\rm in}\,\,H^{2n+2}(X,\mathbb{R})\,\,{\rm if}\,\,q(l)=0\,\,{\rm and}\,\,q(l,h)=0.
\end{eqnarray}
Our application of Theorem \ref{theomain} is the following result, proving that
if $l=c_1(L)$ as above comes from a Lagrangian fibration, $d=c_1(D)$ is a divisor class with $q(l,d)=0$, and the dimension is $\leq 8$, then (\ref{eqrelcoh2})
already holds in ${\rm CH}(X)_\mathbb{Q}$.

\begin{theo} \label{theocons} Let $\phi:X\rightarrow B$ be a Lagrangian fibration on a projective irreducible hyper-K\"ahler manifold of dimension $2n\leq 8$ with $b_2(X)\geq8$ and let
$L$ generate $\phi^*{\rm Pic}\,B$.  Let furthermore
$D\in {\rm Pic}\,X$ satisfy the property that $L^n\cdot D$ is cohomologous to $0$ on $X$. Then
$L^n\cdot D=0\,\,{\rm in}\,\,{\rm CH}^{n+1}(X)_\mathbb{Q}$.
\end{theo}
We refer to \cite{lin} for further applications of Theorem \ref{theocons} to Conjecture \ref{conjbeau}.
We close this introduction by mentioning  that there are other  potential applications of the study of Question
\ref{question} in various geometric contexts. Here is an example: Let $S$ be a
$K3$ surface. Huybrechts \cite{huybrechts} defined a constant cycle curve
$C\subset S$ to be a curve all of whose points are rationally equivalent in $S$. All points of $C$ are then
rationally equivalent to the canonical cycle $o_S$.
Rational curves in $S$ are constant cycles curves, but there are many other
constant cycles curves, some of which can be constructed as follows: Choose an ample linear system
$|L|$ on $S$ and let
$\mathcal{C}\subset |L|\times S$ be the universal curve. Over the regular locus
$|L|^0$, we have the Jacobian fibration $\mathcal{J}\rightarrow |L|^0$ and we get
a section
$\nu$ of the pull-back  $\mathcal{J}_\mathcal{C}$ on
$\mathcal{C}$ by defining
$$\nu(c,C)={\rm alb}_C(dc-L_{\mid C}),$$
where $d=L^2$ so that $dc-L_{\mid C}$ has degree $0$ on $C$.
The locus where this  section is of torsion is {\it  expected} to be $1$-dimensional
and to project via the natural map $\mathcal{C}$ to $S$ to a countable union of constant cycles curves. An analysis of Question \ref{question} is necessary here to make this expectation into a statement.
Another related example is as follows: Let now $|L_1|,\,|L_2|$ be two linear systems on the $K3$ surface $S$, and let $N=L_1\cdot L_2$. Choose integers $w_1,\ldots,w_N$ such that $\sum_iw_i=0$.
Over the open set  $(|L_1|\times |L_2|)^0$ where the curves $C_1,\,C_2$ are smooth and meet transversally, there is an \'etale cover $Z\rightarrow (|L_1|\times |L_2|)^0$ parameterizing orderings
of the set $C_1\cap C_2$ and two sections
$$\nu_1:Z\rightarrow \mathcal{J}_{1,Z},\,\nu_2:Z\rightarrow \mathcal{J}_{2,Z},$$
where $\mathcal{J}_{1,Z},\,\mathcal{J}_{2,Z}$ are the pull-backs to $Z$ of the Jacobian fibrations on
$|L_1|$, resp. $|L_2|$. They are defined by
$$\nu_1(z)={\rm alb}_{C_{1,z}}(\sum_iw_i x_{i,z}),\,\nu_2(z)={\rm alb}_{C_{2,z}}(\sum_iw_i x_{i,z}),$$
where $z\in Z$ parameterizes the curves $C_{1,z},\,C_{2,z}$ and the ordering
$C_{1,z}\cap C_{2,z}=(x_{1,z},\ldots,X_{N,z})$ of their intersections.
The vanishing (or torsion)  locus of $(\nu_1,\nu_2)$ is expected to be $0$-dimensional. It provides a geometric way of producing elements of the higher Chow group ${\rm CH}^2(S,1)$ (see \cite{voisinicm}, \cite{collino}, \cite{muellerstach}).

The paper is organized as follows. In the very short section \ref{sec1}, we will show how
Theorem \ref{theocons} follows from Theorem \ref{theomain}. We will prove
Theorem \ref{theoelliptic} in Section \ref{secelliptic}. Section
\ref{secACZ} will be devoted to sketching the proof of the Andr\'e-Corvaja-Zannier theorem.
The proof of Theorem
\ref{theomain} will be given in Section \ref{sec2} in the non-isotrivial case and in Section \ref{sec3} in the isotrivial case. The subsections \ref{subsec}, \ref{secpetitechose} and \ref{secMA} present in detail the local analysis of the nontransversality of a Lagrangian section of a Lagrangian torus fibration and show how it is related  to a degenerate real Monge-Amp\`ere equation. In Section \ref{secn5}, we will analyze the case of a $10$-dimensional Lagrangian fibration and prove Proposition \ref{prodim10}.

\vspace{0.5cm}

{\bf Thanks}. I thank S\'ebastien Boucksom and Jean-Pierre Demailly for useful discussions
concerning the degenerate Monge-Amp\`ere equation. I also thank Pietro Corvaja and Umberto Zannier
for  introducing me to the number theorist terminology of Betti coordinates,
for providing a guide through the existing litterature, for communicating their
 paper \cite{ACZ} and for explaining to me their results. The present  paper is a completely new version of the    paper arXiv:1603.04320, and the use of  their main theorem \ref{theoandre} led me to a much  improved result.
\section{Application of Theorem \ref{theomain} to the Beauville conjecture\label{sec1}}
We prove in this section Theorem \ref{theocons} assuming Theorem \ref{theomain}.
Let $\phi:X\rightarrow B$ be as in the introduction, with ${\rm dim}\,X=2n\leq 8$ and $b_2(X)\geq8$ and let
 $L$ generate $\phi^*{\rm Pic}\,B$ (see \cite{matsumain}).
 We want to prove that for any line bundle $D$ on $X$ such that
 $D\cdot L^n$ is cohomologous to $0$ on $X$, then $D\cdot L^n$ is rationally equivalent to $0$ on $X$ modulo torsion.

 Let us consider a universal  family $\mathcal{X}\rightarrow T$ of deformations of $X$ preserving the line bundles
 $L$, $D$ and an ample line bundle on $X$.
 The very general fiber $X_t$ of this family is projective and admits the deformed line bundles
 $D_t$ and $L_t$. It also admits a Lagrangian fibration associated to $L_t$ (see \cite{matsurank},
 and also \cite{geemenvoisin} if the base is not smooth).
 The condition that $D_t\cdot L_t^n$ is rationally equivalent to $0$ on $X_t$ modulo torsion
 is satisfied on a countable union of closed algebraic subsets of $T$, so if we prove it is satisfied
 at the very general point $t\in T$, it will be also satisfied for any $t\in T$, hence for $X$.
 As $b_2(X)\geq8$, we can apply Corollary \ref{corogen} to
   the very general  deformation  $(X_t,\phi_t, L_t,D_t)$.
  It thus suffices to prove the result when
 $(X,\,\phi,\,D)$ satisfies  the conclusion of Theorem \ref{theomain}, which we assume from now on.
In the
Chow ring ${\rm CH}(X)_\mathbb{Q}$, the fibers $X_b$, for $b\in B^0$, are all rationally equivalent (if
$B$ is smooth, then $B\cong \mathbb{P}^n$ by \cite{hwang},  so this is obvious; if $B$ is not smooth, we refer to \cite{lin} for a proof of this statement). It follows that  we have for any $b\in B$
\begin{eqnarray}
\label{forfibre}L^n=\mu X_b\,\,{\rm in}\,\,{\rm CH}^{n}(X)_\mathbb{Q},
\end{eqnarray}
for some nonzero $\mu\in \mathbb{Z}$.
Let $b\in B$ be a general point.
The kernel of the map $$[X_b]\cup=\frac{1}{\mu}c_1(L)^n\cup: H^2(X,\mathbb{Q})\rightarrow H^{2n+2}(X,\mathbb{Q})$$ is equal to the kernel of the restriction map $H^2(X,\mathbb{Q})\rightarrow H^{2}(X_b,\mathbb{Q})$. This is a general fact proved in \cite[Lemme 1.5 and Remarque 1.6]{voisinlag}:
\begin{lemm} (See \cite{voisinlag}) \label{lerestri} Let $j:Y\rightarrow W$ be a  generically finite morphism, where $Y$ and $W$ are
smooth projective varieties (in particular connected), and let
 ${\rm dim}\,W-{\rm dim}\,Y=k$, ${\rm dim}\,W=m$. Let $\alpha:=j_*[Y]_{fund}\in H^{2k}(W,\mathbb{Q})$.
Then
the two $\mathbb{Q}$-vector subspaces
$$K_1:={\rm Ker}\,(j^*:H^2(W,\mathbb{Q})\rightarrow H^2(Y,\mathbb{Q})),\,\,\,K_2:={\rm Ker}\,(\alpha\cup: H^2(W,\mathbb{Q})\rightarrow H^{2k+2}(W,\mathbb{Q}))$$
of $H^2(W,\mathbb{Q})$
are equal.
\end{lemm}

Let now $D\in{\rm Pic}\,X$ such that $D\cdot L^n$ is cohomologous to $0$.
Then $c_1(D)\cup [X_b]=0$ by (\ref{forfibre}), and thus  $c_1(D)_{\mid X_b}=0$ in $H^2(X_b,\mathbb{Q})$,
hence also in   $H^2(X_b,\mathbb{Z})$.
As we assumed that $(X,\phi,D)$ satisfies  the conclusion of Theorem \ref{theomain}, there exist points in $B^0$ such that
$D_{\mid X_b}$ is a torsion point in $ {\rm Pic}^0(X_b)$, so that
$D\cdot X_b$ is a torsion cycle in ${\rm CH}^{n+1}(X)$. This implies by  (\ref{forfibre}) that
$D\cdot L^n$ vanishes in  ${\rm CH}^{n+1}(X)_\mathbb{Q}$, which proves Theorem \ref{theocons}.

\begin{rema}{\rm In our specific case, Lemma \ref{lerestri} also follows
directly from Proposition \ref{promatsu} below, that will be also used later on.}
\end{rema}
\begin{prop}\label{promatsu} (Matsushita \cite{matsurank}) In the situation above, the restriction map $H^2(X,\mathbb{Q})\rightarrow  H^2(X_b,\mathbb{Q})$ has
rank $1$.
 \end{prop}
 Indeed, $K_1$ is always  included in $K_2$, and in the other direction,
if $[X_b]\cup D$ is cohomologous to $0$ in $X$, $D_{\mid X_b}$ cannot be ample, hence it must have trivial
first Chern class by  Proposition \ref{promatsu}.
\section{A toy case: elliptic fibrations \label{secelliptic}}
We will use the same notation
 $\phi:X^0\rightarrow B^0$ for  the restriction to $B^0$ of the elliptic fibration
$\phi: X\rightarrow B$. The associated Jacobian fibration $J\rightarrow B^0$ is a complex manifold which is a
fibration into complex tori, and its sheaf of holomorphic sections
$\mathcal{J}$ is described in complex
analytic terms as the quotient
\begin{eqnarray} \label{eqdeschol} \mathcal{J}=R^1\phi_*\mathcal{O}_{X^0}/R^1\phi_*\mathbb{Z}=R^1\phi_*\mathbb{C}\otimes\mathcal{O}_B/
(\mathcal{H}^{1,0}\oplus R^1\phi_*\mathbb{Z})= \mathcal{H}^{0,1}/R^1\phi_*\mathbb{Z},
\end{eqnarray}
where $$\mathcal{H}^{1,0}=R^0\phi_*\Omega_{X^0/B^0},\,\,\, \mathcal{H}^{0,1}:=R^1\phi_*\mathcal{O}_{X^0}=(R^1\phi_*\mathbb{C}\otimes_{\mathbb{C}}\mathcal{O}_{B^0})/
\mathcal{H}^{1,0}.$$
Formula (\ref{eqdeschol}) describes $J$ as a holomorphic torus fibration but
 as a $\mathcal{C}^\omega$ real torus fibration,
 the right formula
 is
 \begin{eqnarray} \label{eqdesreal}\mathcal{J}_\mathbb{R}=\mathcal{H}^1_\mathbb{R}/
R^1\phi_*\mathbb{Z},\,\,\,\mathcal{H}^1_\mathbb{R}:=R^1\phi_*\mathbb{R}\otimes\mathcal{C}^\omega_{B^0,\mathbb{R}}
\end{eqnarray}
which uses the natural fiberwise isomorphisms
\begin{eqnarray}\label{eqtrivreal} H^1(X_b,\mathbb{R})\cong H^1(X_b,\mathbb{C})/H^{1,0}(X_b)\cong H^1(X_b,\mathcal{O}_{X_b})
\end{eqnarray}
globalizing into the $\mathcal{C}^\omega_{B^0,\mathbb{R}}$ sheaf isomorphism
\begin{eqnarray}\label{eqglobiso4oct} \mathcal{H}^1_\mathbb{R}\cong R^1\phi_*\mathcal{O}_{X^0}\otimes_{\mathcal{O}_B} \mathcal{C}^\omega_{B^0,\mathbb{C}}\cong \mathcal{H}^{0,1}\otimes_{\mathcal{O}_B}\mathcal{C}^\omega_{B^0,\mathbb{C}}.\end{eqnarray}

Trivializing the locally constant sheaves
$R^1\phi_*\mathbb{Z},\,\,R^1\phi_*\mathbb{R}$ on simply connected open sets
$U\subset B$, (\ref{eqdesreal}) is the counterpart,  at the level
of sheaves of sections, of a local real analytic  trivialization
\begin{eqnarray} \label{eqtrivfibR} J_U=U\times (H^{1}(X_0,\mathbb{Z})\otimes \mathbb{R}/\mathbb{Z}),
\end{eqnarray}
where $0$ is a given point of $U$.
Let now $\nu$ be a holomorphic section of $J$ over $U$.
Using the trivialization (\ref{eqtrivfibR}), the section $\nu$ gives a differentiable (in fact real analytic)
map
$$f_\nu: U\rightarrow H^{1}(X_0,\mathbb{Z})\otimes \mathbb{R}/\mathbb{Z}=J(X_0),$$
(which is called the Betti map in \cite{zannier})
and we will use the following  easy criterion:

\begin{prop}\label{ledense} (i) For a  section
$\nu$ of a torus fibration with local associated map
$f_\nu$ as above, the points $x$ of $U$ where $\nu(x)$ is of torsion are dense in
$U$ for the usual topology  if  the differential $df_\nu$ is surjective at some point $b\in U$.

(ii) In the situation of an elliptic fibration over a $1$-dimensional base, the differential $df_\nu$ is surjective if and only if it is nonzero.
\end{prop}
\begin{proof} (i)  Observe first that $f_\nu$ being real analytic, if its differential
is surjective at some point $b\in U$, it is surjective on a dense set of points of $U$, hence it suffices
to prove density near a point $b$ where the differential $df_{\nu,b}$ is surjective.
If the differential $df_{\nu,b}$ is surjective,  $f_\nu$ is an open map
in a neighborhood $U_{b}$ of $b$. As the torsion points are dense in $J(X_{b_0})$, their preimages
under $f_\nu$ are then dense in $U_{b}$.

(ii)  This is a consequence of  the following fact:

\begin{lemm} \label{leker} The kernel of the differential $df_{\nu,b}:T_{U,b}\rightarrow T_{J(X_0)}=H^1(X_0,\mathbb{R})$ is
a complex vector subspace of $T_{U,b}$.
\end{lemm}
\begin{proof} Let us prove more generally that the fibers of $f_\nu$ are analytic subschemes of
$U$.  Let  $\alpha\in J(X_0)=H^1(X_{0},\mathbb{R})/H^1(X_{0},\mathbb{Z})$ and let
$\tilde{\alpha}\in H^1(X_{0},\mathbb{R})$ be a lifting of
$\alpha$ (such a lift is defined up to the addition of an element
 $\beta$ of $H^1(X_{0},\mathbb{Z})$). The class $\tilde{\alpha}$ extends  as a constant section also denoted $\tilde{\alpha}$ of the sheaf
$R^1\phi_*\mathbb{R}$ on $U$, which induces a holomorphic section
$\tilde{\alpha}^{0,1}$ of the sheaf $\mathcal{H}^{0,1}=R^1\phi_*^0\mathcal{O}_{X^0}$ on $U$.
On the other hand, the holomorphic section
$\nu$ of $J$ lifts to a holomorphic section
$\eta$ of $\mathcal{H}^{0,1}$ and it is clear
that $$f_\nu^{-1}(\alpha)=\{b\in U,\,\tilde{\alpha}^{0,1}_b=\eta_b\,\,{\rm in}\,\,\mathcal{H}^{0,1}(X_b)/H^1(X_b,\mathbb{Z})\}.$$
It follows that $f_\nu^{-1}(\alpha)$ is the countable locally finite  union
of the closed analytic subsets defined as zero sets of the holomorphic sections
$$\tilde{\alpha}^{0,1}-\eta
-\beta^{0,1}\in \Gamma(U,\mathcal{H}^{0,1})$$
of the bundle $\mathcal{H}^{0,1}$,
over all sections $\beta$ of $ H^1_\mathbb{Z}$ on $U$.
\end{proof}
In the situation of (ii), the base $U$ and the fiber $J(X_0)$ are both of real dimension $2$ and
Lemma \ref{leker} implies that the kernel of $df_{\nu,b}$ is  of real dimension $0$ or $2$. So either
$df_{\nu,b}=0$ or $df_{\nu,b}$ is surjective.
 \end{proof}
\begin{proof}[Proof of Theorem \ref{theoelliptic}] Let $\phi:X^0\rightarrow B^0$ be our elliptic fibration and  $\nu$ a section of $J(X^0/B)$. Assume that the points of $B$ where
$\nu$ is a torsion point are not dense for the usual topology of $B$. Then by Proposition
\ref{ledense}, it follows that $df_\nu$ vanishes everywhere on any open set  $U$ of $B^0$ where it is defined.
Coming back to the sheaf theoretic language, this means equivalently that
 $\nu$, seen  as a section of  $\mathcal{H}^1_\mathbb{R}/H^1_\mathbb{Z}$ via the isomorphism
 (\ref{eqdesreal}),  is
locally constant, or that our section
$\nu\in \Gamma(B^0,\mathcal{H}^{0,1}/H^1_\mathbb{Z})$ comes from
a locally constant section $\tilde{\nu}_\mathbb{R}\in \Gamma(B^0,{H}^1_\mathbb{R}/H^1_\mathbb{Z})$.
We are thus in position to apply Manin's theorem on the kernel and conclude. For further use,
we give the complete argument here.
Note that,  by assumption, $\tilde{\nu}_\mathbb{R}$ is not of  torsion and thus, fixing a base point
$0\in B$, corresponds to
an element $\alpha_0\in H^1(X_{0},\mathbb{R})$ which is not rational but has  the property that
for any $\gamma\in \pi_1(B^0,0)$,
\begin{eqnarray}\label{eqpourmonoarg}\rho(\gamma)(\alpha_0)-\alpha_0\in H^1(X_{0},\mathbb{Z}),
\end{eqnarray}
where $\rho:\pi_1(B^0,0)\rightarrow {\rm Aut}\,H^1(X_{0},\mathbb{Z})$
denotes the monodromy representation of the smooth fibration $\phi: X^0\rightarrow B^0$.
We use the following easy lemma.
\begin{lemm}\label{letorsionmono} Let $\rho:\Gamma\rightarrow {\rm Aut}\,V_\mathbb{Q}$
be a finite dimensional rational representation of a group $\Gamma$.
Then if the invariant space $V_\mathbb{Q}^{inv}$ is trivial, the set
$$\{v\in V_\mathbb{R},\,\,
\rho(\gamma)(v)-v\in V_\mathbb{Q}{\rm \,\,\,
for \,\,any \,\,}\gamma\in \Gamma\,\,\}$$
is equal to $V_\mathbb{Q}$.
\end{lemm}
As the monodromy representation is rational, this lemma tells us that the
set of classes
$\alpha \in H^1(X_{0},\mathbb{R})$ satisfying property (\ref{eqpourmonoarg}) contains
a nonrational class if and only if the set
$$H^1(X_{0},\mathbb{Q})^{inv}:=\{\alpha\in  H^1(X_{0},\mathbb{Q}),\,\rho(\gamma)(\alpha)-\alpha=0,
\,\,\forall \gamma\in \pi_1(B^0,0)\}$$
of monodromy invariant elements is nontrivial.
By Deligne's invariant cycles theorem \cite{deligne},
it then follows from the existence of $\alpha_0$ that the restriction map
$$H^1(X,\mathbb{Q})\rightarrow H^1(X_{0},\mathbb{Q})$$
is nontrivial, hence surjective since this a morphism of Hodge structures of weight $1$ and the right hand side has dimension $2$.
The conclusion that $X$ is  rationally isogenous
to $J(X_{0})\times B$ is then immediate since $X$ is rationally isogenous
to a projective completion of the Jacobian fibration
$J(X^0/B^0)$ and the later is isogenous to $J(X_{0})\times B^0$ if the
restriction map $H^1(X,\mathbb{Q})\rightarrow H^1(X_{0},\mathbb{Q})$ is surjective.
\end{proof}
\subsection{Some examples with higher dimensional fiber dimension}
Looking at Theorem \ref{theoelliptic}, one may wonder
 what makes hard and wrong  a generalization to
  higher fiber dimension. Let us give some examples explaining the main difficulties
 encountered:
Consider more generally any complex torus fibration $\phi_A:A\rightarrow B$ with a section
$\nu_A$. The torus fibration is canonically isomorphic, as a real torus fibration, to the locally constant fibration
$H_{1,\mathbb{R},\phi_A}/H_{1,\mathbb{Z},\phi_A}$, where
$H_{1,\mathbb{R},\phi_A}:=(R^1\phi_{A*}\mathbb{R})^*$. A section $\nu_A$ thus admits
local liftings $\tilde{\nu}_{A,\mathbb{R}}$ which are $\mathcal{C}^\infty$ (in fact real analytic) sections of
the flat vector bundle associated with the local system $H_{1,\mathbb{R},\phi_A}$.
\begin{example}\label{ex05}  Assume that the local lifts $\tilde{\nu}_{A,\mathbb{R}}$
of the section $\nu_A$ are     constant
sections $\tilde{\nu}_{A,\mathbb{R}}\in \Gamma(H_{1,\mathbb{R},\phi_A})$. Then if $\tilde{\nu}_{A,\mathbb{R}}$ is not rational, that is, does not belong to $\Gamma(H_{1,\mathbb{Q},\phi_A})$, $\nu_A$ has no torsion point.
\end{example}
This case, which corresponds to the situation where the local maps $f_{\nu_A}$ of the previous section are constant, can be in general excluded by a monodromy argument. The following example is slightly more subtle:

\begin{example} \label{ex2} Assume that the torus  fibration
$\phi_A:A\rightarrow B$ is  a fibered product
$$A=A'\times_B A'',$$
where $\phi':A'\rightarrow B,\,\phi'':A''\rightarrow B$ are  torus fibrations.
Choose a section  $\nu_{A''}:B\rightarrow  A''$ of $\phi''$ which is as in Example \ref{ex05}, that is locally lifts  to a constant nonrational
section of $(R^1\phi''_*\mathbb{R})^*$. Then for any section $\nu_{A'}:B\rightarrow A'$ of $\phi'$,
$\nu_A=(\nu_{A'},\nu_{A''})$ does not have any torsion point.
\end{example}

Another slightly more involved situation where we can avoid torsion points is as follows: Recall from the introduction that if ${\rm dim}\,B< g:={\rm dim} \,A_b$
a ``generic '' section of the abelian scheme $\phi_A:A\rightarrow B$ avoids the torsion points. Of course, for a general family $A\rightarrow B$, there might be no non trivial section, but there are always multisections and
they can be chosen to pass through any point with arbitrary differential. By ``generic'' we mean generic section of a base-changed family $A'\rightarrow B'$.
The next example exhibiting nontransversality is as follows:
\begin{example} Choose $A_1\rightarrow B_1$ with ${\rm dim}\, B_1<g_1 :={\rm dim} \,A_{1,b}$ and  a section $\nu_1:B_1\rightarrow A_1$ with no torsion points. For any $A_2\rightarrow B_2$, such that ${\rm dim}\, B_1+{\rm dim}\, B_2=g_1+g_2$, and for any section
$\nu_2:B_2\rightarrow A_2$, the section $(\nu_1,\nu_2):B_1\times B_2\rightarrow A_1\times A_2$ of the product family $A_1\times A_2\rightarrow B_1\times B_2$ has no torsion point.
\end{example}

Let us conclude with  the following abstract situation where we do not have transversality of the normal function (or rather of its local real analytic representation as in Section \ref{secelliptic}, so we do not expect the normal function to have torsion points.
\begin{example} \label{exampleimprobable} Assume the abelian scheme $A\rightarrow B$ satisfies ${\rm dim}\,B=g={\rm dim}\,A_b$ and  has the following property: $B$ admits a  foliation  given by an algebraic distribution which is analytically integrable with holomorphic leaves   $\mathcal{F}_t$ of dimension
$d$, such that along the leaves, the
real variation of Hodge structures on $H_{1\mid\mathcal{F}_t} $ splits as
$H_{1,\mathbb{R}\mid\mathcal{F}_t}=L_{\mathbb{R},t}\oplus L'_{\mathbb{R},t}$, with ${\rm rank}\,L'_\mathbb{R}=2d'$, and $d'<d$. Then if furthermore, the normal function
$\nu:B\rightarrow A$ or rather its local real analytic lift
$\tilde{\nu}_{\mathbb{R}}$ decomposes along the leaves as
$\nu_{\mathbb{R}\mid \mathcal{F}_t}=\nu_{L,t}+\nu_{L',t}$ with $\nu_{L,t}$ locally constant, then $\nu_{\mathbb{R}}$ is never of maximal rank.
\end{example}
Indeed, the differential $d\nu_{\mathbb{R}}$ cannot be injective at any point since its restriction to the tangent space of the leaf  is equal to
$d\nu_{L',t}:T_{\mathcal{F}_{t},\mathbb{R}}\rightarrow L'_{\mathbb{R},t}$, the two spaces being of dimensions $2d>2d'$.

This situation seems to be improbable if the $L_{\mathbb{R},t}$ furthermore varies with $t$ (that is, does not come from a local system on $B$). However this is the only case that we could not exclude for normal sections associated with divisors on hyper-K\"ahler manifolds of dimension $10$.
Let $X\rightarrow B$ be a projective  hyper-K\"ahler manifold
of dimension $10$ equipped with a Lagrangian fibration with maximal variation, and let $M\in {\rm Pic}\,X$
be a divisor which is cohomologous to $0$ on fibers. The following result will be proved in Section \ref{secn5}:
\begin{prop}\label{prodim10} (i) If the torsion points of $\nu_M: B^0\rightarrow X^0$ are not dense in $B$ for the usual topology, a Zariski dense open set $B^0$ has a foliation with $3$-dimensional  leaves $\mathcal{F}_t$, and the variation of Hodge structure along the leaves
$$H_{1,\mathbb{R}\mid\mathcal{F}_t}=L_{\mathbb{R},t}\oplus L'_{\mathbb{R},t}$$
where ${\rm rank}\,L_{\mathbb{R},t}=6,\,{\rm rank}\,L'_{\mathbb{R},t}=4$.
 Furthermore, the real variation of Hodge structure on $L_{\mathbb{R},t}$ is constant along $\mathcal{F}_t$.

 (ii) Along each leaf $\mathcal{F}_t$, the normal function $\nu_M$ has the $L$-component $\tilde{\nu}_{M,\mathbb{R},L}$ of its real lift constant.
\end{prop}
\section{The Andr\'e-Corvaja-Zannier theorem \label{secACZ}}
We describe in this section following \cite{ACZ} the proof of Theorem
\ref{theoandre}. The reason for doing so is not only the fact that this result is important and
 the arguments beautiful, but also the fact that their paper is written with notations
 that are not familiar to Hodge theoretists.

We first comment on the meaning of the conclusion of the theorem and the notation used there: consider the local system
$H_{1,\mathbb{Z}}=(R^1\pi_*\mathbb{Z})^*$ on $B$. It  generates  the flat holomorphic vector bundle
$\mathcal{H}_1:=H_{1,\mathbb{Z}}\otimes \mathcal{O}_B$, which carries the Hodge filtration
$$\mathcal{H}_{1,0}\subset \mathcal{H}_1$$
with quotient $\mathcal{H}_{0,1}$. The Griffiths $\overline{\nabla}$ map
is defined as the composite
$$\mathcal{H}_{1,0}\stackrel{\nabla}{\rightarrow} \mathcal{H}_1\otimes \Omega_B\rightarrow
\mathcal{H}_{0,1}\otimes \Omega_B,$$
where $\nabla$ is the Gauss-Manin connection.
It is $\mathcal{O}_B$-linear, hence induces, for each
$b\in B,\,\lambda\in \mathcal{H}_{1,0,b}$, a linear map:
$\overline{\nabla}_\lambda:T_{B,b}\rightarrow \mathcal{H}_{0,1,b}$.

 A crucial ingredient in the proof of Theorem \ref{theoandre} is the following result due to Andr\'e
 \cite{andremono}:
 In the situation of the theorem, let $\Gamma_0\subset
 \pi_1(B,b)$ be the subgroup acting trivially on $H^1(A_b,\mathbb{Z})$. On the cover
 $B_{\Gamma_0}\rightarrow B$ equipped with a point $\tilde{b}$ over $b\in B$, the base-changed fibration $A_{\Gamma_0}\rightarrow B_{\Gamma_0}$ has a natural globally defined real analytic trivialization $A_{\Gamma_0}\cong B\times H_1(A_{\tilde{b}},\mathbb{R})/H_1(A_{\tilde{b}},\mathbb{Z})$, so that the
 section $\nu$ becomes a well-defined real analytic map ${f}_\nu:B_{\Gamma_0}\rightarrow H_1(A_{\tilde{b}},\mathbb{R})/H_1(A_{\tilde{b}},\mathbb{Z})$.
 \begin{theo} \label{theoandremono} (Andr\'e \cite{andremono}) Under the assumptions (i) and (ii), the image of the  monodromy
 $$\Gamma_0\rightarrow H_1(A_{\tilde{b}},\mathbb{Z})$$
 of
  ${f}_\nu$ is Zariski dense in $H_1(A_{\tilde{b}},\mathbb{C})$.
  \end{theo}
\begin{proof}[Proof of Theorem \ref{theoandre}] We work on $B_{\Gamma_0}$. The real analytic map
${f}_\nu$ is constructed as follows:
The holomorphic section $\nu$ on $B$ is a section of the sheaf
$\mathcal{H}_{0,1}/H_{1,\mathbb{Z}}$. Choose a local lift $\nu_{0,1}\in\Gamma( \mathcal{H}_{0,1})$.
 Due to Hodge symmetry, the real analytic  flat vector bundle
$\mathcal{H}_{1,\mathbb{R}}$ is isomorphic to $\mathcal{H}_{0,1}$ as a real analytic vector bundle.
This way, the holomorphic section $\nu_{0,1}$ of $\mathcal{H}_{0,1}$ provides a real analytic section
$\tilde{\nu}_\mathbb{R}$ of $\mathcal{H}_{1,\mathbb{R}}$.
We need to describe more explicitly how $\tilde{\nu}_\mathbb{R}$ is deduced from
$\nu_{0,1}$. Choose a local holomorphic lift $\tilde{\nu}_{hol}$ of $\nu$ to
a holomorphic section of the bundle $\mathcal{H}_1$.
We use now the real analytic decomposition
\begin{eqnarray}\label{eqdecophodge} \mathcal{H}_{1,an}\cong \mathcal{H}_{1,0,an}\oplus \mathcal{H}_{0,1,an},\end{eqnarray}
with $\mathcal{H}_{0,1}=\overline{\mathcal{H}_{1,0}}$. Here the index ``an'' indicates that we consider the associated real analytic complex vector bundle. We can thus  write
$\tilde{\nu}_{hol}-\overline{\tilde{\nu}_{hol}}=\eta_1-\overline{\eta_1}$ for some real analytic section $\eta_1$
of $\mathcal{H}_{1,0}$.
Then we conclude that
\begin{eqnarray}\label{eqdecomp} \tilde{\nu}_{hol}-\eta_1=\overline{\tilde{\nu}_{hol}}-\overline{\eta_1}
\end{eqnarray}
is real and maps to ${\nu}_{0,1}$ via the quotient
map $ \mathcal{H}_{1,an}\rightarrow \mathcal{H}_{0,1,an}$. Thus $\tilde{\nu}_{hol}-\eta_1$ gives the desired
section $\tilde{\nu}_\mathbb{R}$ of $ \mathcal{H}_{1,\mathbb{R},an}$ (which is then made into a real analytic map
$f_\nu:B\rightarrow H_1(A_b,\mathbb{R})$ by local flat trivialization of $ \mathcal{H}_{1,\mathbb{R},an}$). Note that it is clearly independent of the choice of holomorphic lifting $\tilde{\nu}_{hol}$ of $\nu_{0,1}$.

Our assumption is that $f_\nu$ is nowhere of maximal rank. What is unpleasant about $f_\nu$ is its real analytic, as opposed to holomorphic, character.
The first step in the Andr\'e-Corvaja-Zannier proof is the separation of holomorphic and antiholomorphic variables so as to make this assumption into a holomorphic equation.
Let us work locally on $B\times \overline{B}$ (later on, we will rather consider
$B_{\Gamma_0}\times \overline{B_{\Gamma_0}}$). Here $\overline{B}$ is $B$ equipped with its conjugate
complex structure, for which the holomorphic functions are the complex conjugates of holomorphic functions on
$B$. Over $\overline{B}$ we have the holomorphic bundle $\overline{\mathcal{H}_{1,0}}$ and
this provides us with two holomorphic vector bundles
$$\mathcal{K}_1:=pr_1^*\mathcal{H}_{1,0},\,\,\mathcal{K}_2:=pr_2^*\overline{\mathcal{H}_{1,0}}$$
on $B\times\overline{B}$.
Near the diagonal of $B$, we have the flat holomorphic bundle $\mathcal{H}_1$ associated
to the local system $H_{1,\mathbb{C}}$ on $B$. The inclusions of holomorphic subbundles
$$\mathcal{H}_{1,0}\subset \mathcal{H}_1,\,\,{\rm resp.}\,\, \overline{\mathcal{H}_{1,0}}\subset \mathcal{H}_1$$
on $B$, resp. $\overline{B}$
gives  on $B\times \overline{B}$  two holomorphic vector subbundles
  $$\mathcal{K}_1\subset \mathcal{H}_1 ,\,\, \mathcal{K}_2 \subset \mathcal{H}_1,$$
  which restricted to the diagonal produce (\ref{eqdecophodge}). Thus
  we have
  \begin{eqnarray}\label{eqhodebxb} \mathcal{H}_1=\mathcal{K}_1\oplus \mathcal{K}_2
  \end{eqnarray}
near the diagonal of $B$ in $B\times\overline{B}$.
  We now produce a holomorphic version of $\tilde{\nu}_{\mathbb{R}}$ on $B\times \overline{B}$ near the diagonal as follows: starting from  local  lifts $\tilde{\nu}_{hol}$ of $\nu_{0,1}$ to a holomorphic section of $\mathcal{H}_1$ on $B$, we consider $$\tilde{\nu}_1:=pr_1^*\tilde{\nu}_{hol},\,\,\tilde{\nu}_2=pr_2^*\overline{\tilde{\nu}_{hol}}$$ as holomorphic
  sections
  of  $\mathcal{H}_1$ defined on $B\times\overline{B}$ near the diagonal
  and write
  $\tilde{\nu}_1-\tilde{\nu}_2=\eta_1-\eta_2$, where $\eta_1,\,\eta_2$ are holomorphic
  sections of $\mathcal{K}_1,\,\mathcal{K}_2$ respectively.
  We then consider the holomorphic section
  $\tilde{\nu}_1-\eta_1$ of $\mathcal{H}_1$ and observe that its restriction to the diagonal of $B$ is exactly our real lifting $\tilde{\nu}_\mathbb{R}$. Similarly, the  locally defined holomorphic map
  $F_\nu: B\times \overline{B}\rightarrow H^1(A_b,\mathbb{C})$ deduced from $\tilde{\nu}_1-\eta_1$ by locally trivializing the flat vector bundle $\mathcal{H}_1$,
  restricts to $f_\nu:B\rightarrow H^1(A_b,\mathbb{R})$
  on the diagonal $\Delta_B$.

  One easily shows (see \cite[Lemma 4.2.1]{ACZ}) that $F_\nu:B\times\overline{B}\rightarrow H^1(A_b,\mathbb{C})$ is  generically of maximal rank if and only if $f_\nu: B\rightarrow H^1(A_b,\mathbb{R})$ is generically of maximal rank. Thus the assumption of the theorem  is that $F_\nu$ is nowhere
  of maximal rank where it is defined, namely near the diagonal of $B$.

  It is important to observe that $\tilde{\nu}_1-\eta_1$ and $F_\nu$ in fact do not depend on the choice of lift $\tilde{\nu}$ of $\nu_{0,1}$. Indeed, if we add
  to $\tilde{\nu}$ a section $\lambda$ of $\mathcal{H}_{1,0}$ without changing  $\overline{\tilde{\nu}}$,
  then
  we have
  $\tilde{\nu}'_1=\tilde{\nu}_1+\lambda_1$, with $\lambda_1=pr_1^*\lambda$, and
  $\tilde{\nu}'_2=\tilde{\nu}_2$, so that $\tilde{\nu}'_1-\tilde{\nu}'_2=\tilde{\nu}_1-\tilde{\nu}_2+\lambda_1$ and $\eta'_1=\eta_1+\lambda_1$.
  Thus
  $\tilde{\nu}'_1-\eta'_1=\tilde{\nu}_1-\eta_1$. Similarly, if we change now $\overline{\tilde{\nu}}$
  by adding to it a section $\overline{\lambda}$ of $\mathcal{H}_{0,1}$,
  then we do not change $\tilde{\nu}'_1$ and we have $\tilde{\nu}'_2=\tilde{\nu}_2+\lambda_2$
  where $\lambda_2=pr_2^*\overline{\lambda}$. Then $\eta'_1=\eta_1$ and $\eta'_2=\eta_2+\lambda_2$, so finally $\tilde{\nu}'_1-\eta'_1=\tilde{\nu}_1-\eta_1$.

  In order to use  the monodromy Theorem \ref{theoandremono}, we need to make the above construction
  more global, and this can be done by working on
  $B':=B_{\Gamma_0}\times \overline{B_{\Gamma_0}}$. On $B_{\Gamma_0}$, the local system
  ${H}_1$ is trivial, and thus we have a global isomorphism
  $pr_1^*H_1\cong pr_2^*H_1=H$, so that $\mathcal{K}_1=pr_1^*\mathcal{H}_{1,0}$ and
  $\mathcal{K}_2=pr_2^*\overline{\mathcal{H}_{1,0}}$ are holomorphic subbundles
  of the trivial bundle $\mathcal{H}:=H\otimes\mathcal{O}_{B'}$.
  On a dense   Zariski open set $B'_0$ of $B'$, we have
  $\mathcal{H}\cong \mathcal{K}_1\oplus \mathcal{K}_2$.
  Furthermore $pr_1^*\mathcal{H}_{0,1}\cong \mathcal{H}/\mathcal{K}_1$.
   We thus get a multivalued holomorphic map on $B'_0$
  \begin{eqnarray}\label{eqnamefnu} F_\nu:B'_0\rightarrow H_1(A_{\tilde{b}},\mathbb{C})
  \end{eqnarray}
  which coincides with the previously defined map near the diagonal
  of $B_{\Gamma_0}$. The fact that $F_\nu$ is multivalued follows from the fact that
  it is well defined once the lifts $\nu_{0,1},\,\overline{\nu_{0,1}}$ are chosen, but it depends on these lifts.

We now use Theorem \ref{theoandremono} on  $B_{\Gamma_0}$. It says that the couple
$(\nu_{0,1},\overline{\nu_{0,1}})$ (or their lifts $\tilde{\nu},\,\overline{\tilde{\nu}}$) changes under monodromy along  $B'$ (or $B'_0$) by the addition of couples
$(u,v)\in H_1(A_{\tilde{b}},\mathbb{Z})^2$ with $u,\,v$ in a Zariski dense open
set of $H_1(A_{\tilde{b}},\mathbb{Z})^2\subset H_1(A_{\tilde{b}},\mathbb{C})^2$.
Write now
$$\tilde{\nu}_u=\tilde{\nu}+u,\,\overline{\tilde{\nu}}_v=\overline{\tilde{\nu}}+v.$$
Then, using (\ref{eqhodebxb}), we can write  $u-v=\lambda_1-\lambda_2$ on $B'_0$, where
$\lambda_1$ is  a holomorphic section of $\mathcal{K}_1$ and
$\lambda_2$ is  a holomorphic section of $\mathcal{K}_2$, and
we get
$$pr_1^*\tilde{\nu}_{u}-pr_2^*\overline{\tilde{\nu}}_v=
pr_1^*\tilde{\nu}_{0,1}-pr_2^*\overline{\tilde{\nu}_{0,1}}+u-v$$
$$=pr_1^*\tilde{\nu}_{0,1}-pr_2^*\overline{\tilde{\nu}_{0,1}}+\lambda_1-\lambda_2=\eta_1-\eta_2+
\lambda_1-\lambda_2,$$
so that $\tilde{\nu}_{1,u,v}=\tilde{\nu}_{1}+u$ and $\eta_{1,u,v}=\eta_1+\lambda_1$.
According to the recipe described above,
the map $F_\nu$  thus becomes
$F_\nu+u-\lambda_1$. A translation by $u$ does not change the differential of $F_\nu$, so we can do here $u=0$ and choose $v$ in a Zariski open set
of $H^1(A_{\tilde{b}},\mathbb{Z})$.
As the holomorphic map $F_\nu$ is nowhere of maximal rank,  we conclude that for any $v$ in a Zariski dense set of $H^1(A_{\tilde{b}},\mathbb{Z})\subset H^1(A_{\tilde{b}},\mathbb{C})$, the holomorphic map
$F_\nu-\lambda_1:B'_0\rightarrow H^1(A_{\tilde{b}},\mathbb{C})$ is nowhere of maximal rank, where $\lambda_1$ is the holomorphic section
of
$\mathcal{K}_1\subset \mathcal{H}_1$ defined by
$$v=\lambda_1-\lambda_2.$$
We  now apply the following easy lemma: for any
$\mu\in H^1(A_{\tilde{b}},\mathbb{C})$, we can write as before, using (\ref{eqhodebxb})
$$\mu=\mu_1+\mu_2$$
as sections of $\mathcal{H}$ on $B'_0$, with $\mu_i$   a holomorphic section of $\mathcal{K}_i$, $i=1,\,2$. We can see $\mu_1$ as a holomorphic map $B'_0\rightarrow H=H_1(A_{\tilde{b}},\mathbb{C})$.
\begin{lemm} \label{ledensepourandre} Fix $b\in B'_0$. Then the set of $\mu \in H_1(A_{b'},\mathbb{C})$
such that the holomorphic map $F_\nu-\mu_1$ is not of maximal rank at $b$ is
Zariski closed in $H_1(A_{b'},\mathbb{C})$.
\end{lemm}
\begin{proof} Indeed, the assigment
$\mu\mapsto \mu_1$ is $\mathbb{C}$-linear
in $\mu$, hence for fixed $b$, $\mu_1,\,d\mu_{1,b}$ depend $\mathbb{C}$-linearly on $\mu$, which concludes the proof.
\end{proof}
Lemma \ref{ledensepourandre} and the fact that $v$ above can be taken in a Zariski dense set of $H^1(A_{\tilde{b}},\mathbb{Z})\subset H^1(A_{\tilde{b}},\mathbb{C})$
allow now to conclude that
for any $b'\in B'$ and any $\mu\in  H_1(A_{b},\mathbb{C})$, the holomorphic map
$F_\nu-\mu_1$ is not of maximal rank at $b'$ so that
$\mu_1$ is not of maximal rank at $b'$.
The proof of Theorem \ref{theoandre} concludes now with the following lemma:
\begin{lemm} Assume that for any $b'\in B'_0$, and any  $\mu\in H_1(A_{b'},\mathbb{C})$,
the differential $d\mu_1:T_{B',b'}\rightarrow H_1(A_{b'},\mathbb{C})$
is not of maximal rank.
Then for any $b\in B$, and any
$\lambda\in H_{1,0}(A_b)$ the map
$\overline{\nabla}_\lambda:T_{B,b}\rightarrow H_{0,1,b}$ is not of maximal rank.
\end{lemm}
\begin{proof} The subbundle $\mathcal{K}_1\subset \mathcal{H}$, where $\mathcal{H}$ is trivial, has a variation
$\delta_1:\mathcal{K}_1\rightarrow \mathcal{K}_2\otimes\Omega_{B'}$ and similarly for
$\mathcal{K}_2$. Along $B\subset B\times\overline{B}$, $\delta_1=\overline{\nabla}$
and $\delta_2$ is its complex conjugate.
The differential $d\mu_1$ is computed as follows:
we have \begin{eqnarray}\label{eqdmuzero} d\mu=0=d\mu_1+d\mu_2, \end{eqnarray}with
$$d\mu_{1}=d_1\mu_1+\delta_1(\mu_1),\,d\mu_2=d_2\mu_2+\delta_2(\mu_2)$$
for some differentials $d_i\mu_i\in \Omega_{B'}\otimes \mathcal{K}_i$.
It follows from (\ref{eqdmuzero}) that $d_1\mu_1+\delta_2(\mu_2)=0,\,d_2\mu_2+\delta_1(\mu_1)=0$, so that
\begin{eqnarray}\label{eqpourdifffinale} d\mu_1=-\delta_2(\mu_2)+\delta_1(\mu_1).
\end{eqnarray}
Suppose there exist  $b\in B$  and
$\lambda\in H^{1,0}(A_b)$ such that the map
$\overline{\nabla}(\lambda):T_{B,b}\rightarrow H_{0,1,b}$ is  of maximal rank.
Let $\mu=\lambda+\overline{\lambda}$.
Equation (\ref{eqpourdifffinale}) at  $b\in B\subset B\times\overline{B}$
gives then
$ d\mu_1=-\overline{\overline{\nabla}(\lambda)}+\overline{\nabla}(\lambda)$
and this sum is the direct sum of the two isomorphisms
$$\overline{\nabla}(\lambda):T_{B}\cong H_{0,1}(A_b),\,-\overline{\overline{\nabla}(\lambda)}:T_{\overline{B}}\cong H_{1,0}(A_b),$$
hence it is an isomorphism
$$d\mu_1:T_{B',b}=T_{B}\oplus T_{\overline{B}}\cong H_1(A_b,\mathbb{C}),$$
contradicting our assumption.
\end{proof}
This concludes the proof of Theorem \ref{theoandre}.
\end{proof}
We finish this section by observing that the proof given above allows to
prove in fact a  statement slightly stronger than
Theorem \ref{theoandre}, namely the following variant, for which we introduce a notation:
associated to our normal function $\nu$, which is a section of the sheaf
$$\mathcal{J}=\mathcal{H}_{0,1}/H_{1,\mathbb{Z}}$$
and given a local lift of $\nu$ to a holomorphic section $\nu_{0,1}$
of $\mathcal{H}_{0,1}$, we get an affine subbundle $\mathcal{H}_{1,0,\nu}$ of
$\mathcal{H}_1=H_{1,\mathbb{Z}}\otimes \mathcal{O}_B$ defined as
\begin{eqnarray}
\label{eqh10nu}\mathcal{H}_{1,0,\nu}:=q^{-1}(\nu_{0,1}),
\end{eqnarray}
where $q:\mathcal{H}_1\rightarrow \mathcal{H}_{0,1}$ is the quotient map.
Composing the inclusion $\mathcal{H}_{1,0,\nu}\subset \mathcal{H}_1$ of the affine subbundle
$\mathcal{H}_{1,0,\nu}$ with a local flat trivialization
$\Phi:\mathcal{H}_1\rightarrow H_{1,b,\mathbb{C}}$, we get a holomorphic map
\begin{eqnarray}\label{eqPhinu}\Phi_\nu: \mathcal{H}_{1,0,\nu}\rightarrow H_{1,b,\mathbb{C}}.
\end{eqnarray}

\begin{theo}\label{theoandrevariant} Let $\pi: A\rightarrow B$ be a family of abelian varieties of dimension $d$ over $\mathbb{C}$, where $B$ is quasi-projective and ${\rm dim}\,B\geq d$.
Let $\nu$ be an algebraic section of $\pi$. Assume

 (i) The family has no locally trivial subfamily.

  (ii) The multiples $m \nu(B)$ are Zariski dense in $A$.

  Then if the associated real analytic map $f_\nu$  is nowhere of maximal rank, the map $\Phi_\nu$ is nowhere submersive.

\end{theo}
Note that this strengthens Theorem \ref{theoandre} since
the conclusion of Theorem \ref{theoandre} is the fact that the holomorphic map
$\Phi_{1,0}:=\Phi_{\mid \mathcal{H}_{1,0}}:\mathcal{H}_{1,0}\rightarrow H_{1,b,\mathbb{C}}$ is nowhere of maximal rank.
Choosing a holomorphic lift
$\tilde{\nu}$ of $\nu$ in $\mathcal{H}_1$,
the map $\Phi_\nu$ of (\ref{eqPhinu}) identifies to
$\Phi\circ \tilde{\nu}+\Phi_{1,0}:\mathcal{H}_{1,0}\rightarrow H_{1,b,\mathbb{C}}$ and it is clear, by passing to the linear part, that
if  $\Phi\circ \tilde{\nu} +\Phi_{0,1}$ is nowhere of maximal rank on $\mathcal{H}_{1,0}$, so is
$\Phi_{1,0}$.

\section{Lagrangian fibrations on hyper-K\"ahler manifolds; the non-isotrivial case of Theorem \ref{theomain}\label{sec2}}

\subsection{\label{subsec}Local structure results}
We  recall in this section, following
\cite{donagimarkaman} or \cite{syz}, the local data determining a holomorphic Lagrangian polarized torus fibration.
The holomorphic torus fibration $\phi:A\rightarrow U$ is determined by a weight $1$ (or rather $-1$) variation of Hodge structure, that is the data of
a local  system $H_{1,\mathbb{Z}}=(R^1\phi_*\mathbb{Z})^*$ and a holomorphic subbundle
$$\mathcal{H}_{1,0}\subset \mathcal{H}_1:=H_{1,\mathbb{Z}}\otimes \mathcal{O}_U$$
determining a weight $1$ Hodge structure at any point of $U$.
The sheaf $\mathcal{A}$ of holomorphic sections of $A$ identifies to
\begin{eqnarray}
\label{eqsheafsec} \mathcal{H}_{0,1}/H_{1,\mathbb{Z}},\,\,\mathcal{H}_{0,1}:=\mathcal{H}_1/\mathcal{H}_{1,0} .
\end{eqnarray}

The holomorphic $2$-form $\sigma_A$ on $A$ for which the fibration
$\phi:A\rightarrow U$ is a Lagrangian fibration provides by interior product an isomorphism of vector bundles
\begin{eqnarray}\label{eqisovbundles} T_U=R^0\phi_*(T_{A}/T_{A/U})\cong R^0\phi_*\Omega_{A/U},
\end{eqnarray}
which by dualization provides an isomorphism (which is canonical, given the choice of $\sigma_A$)
\begin{eqnarray}\label{eqisovbundlesdual} \mathcal{H}_{0,1}\cong \Omega_U.
\end{eqnarray}
Using the isomorphism of (\ref{eqisovbundlesdual}), the surjective map of holomorphic vector bundles
$$\mathcal{H}_1\rightarrow \mathcal{H}_{0,1}$$
is thus given, on any simply connected open set $U'$ where the flat vector bundle $\mathcal{H}_1$
is trivialized, by the evaluation morphism of  $2n$ holomorphic $1$-forms
$\alpha_i$ on $U'$, which must have the property
that their real parts ${\rm Re}\,\alpha_i$ form a basis of
$\Omega_{U',\mathbb{R}}$ at any point.
We have  (see \cite{donagimarkaman}):
\begin{lemm} \label{leformalphaiclosed} The forms $\alpha_i$ are closed on $U'$, hence we have, up to shrinking
our local open set $U'$ if necessary,  $\alpha_i=df_i$, where the $f_i$'s are holomorphic and defined up to a constant.
\end{lemm}
\begin{proof}
The proof will use  a different description of
the forms
$\alpha_i$. For any
locally constant section
$\gamma$ of $H_{1,\mathbb{C}}\cong H^{2g-1}_{\mathbb{C}}$, $g={\rm dim}\,A_b$, we get using the closed
$2$-form $\sigma_A$ a closed $1$-form
$\phi_{*}( \gamma\cup  \sigma_A)$ on $U'$ (when $\gamma$ is a section of  $H_{1,\mathbb{Z}}$,
$\gamma$ is a combination of classes $\gamma_i$ of oriented  circle bundles
$\Gamma_i\subset A$, and $\phi_{*}(\gamma\cup \sigma_A)$ is defined as
$(\phi_{\mid \Gamma_i})_*(\sigma_{A\mid \Gamma_i})$).
We conclude observing that
for any locally constant section
$\gamma$ of $H_{1,\mathbb{C}}$ on an open set $U'$ of $U$, with induced  section
${\gamma}_{0,1}$ of $\mathcal{H}_{0,1}$, providing a holomorphic form
$\alpha_{\gamma_{0,1}}$ via the isomorphism (\ref{eqisovbundlesdual}), we have
$$\alpha_{\gamma_{0,1}}=\phi_{*}( \gamma\cup \sigma_A)\,\,{\rm in}\,\,\Gamma(U',\Omega_{U'}),$$
thus proving that the forms $\alpha_i$ are closed.
\end{proof}
If we now choose the $\alpha_i$ to form a basis of
$H_{1,\mathbb{R}}$, the corresponding $2n$  holomorphic $1$-forms
on $U'$ have the properties that
at any point $b\in U'$, the forms $\alpha_{i,b}$ are independent over
$\mathbb{R}$. Another way to say it is that the functions ${\rm Re}\,f_i$ give
local real coordinates on $U'$.
\begin{rema}{\rm By definition, the functions ${\rm Re}\,f_i$, which are defined up to a
 constant and depend only on the choice of $\sigma_A$ and of local  basis
 $\alpha_i$ of $H_{1,\mathbb{R}}$, provide a local real analytic identification
\begin{eqnarray}
\label{eqlocalident} U\cong H^1(A_{b_0},\mathbb{R}).
\end{eqnarray}
Globally, they provide an affine flat structure on $U$.
}
\end{rema}

The last piece of information we need concerning the torus fibration $\phi:A\rightarrow U$  is the polarization $\omega$ on the fibers.
It is given by a monodromy invariant skew-symmetric  pairing
$\langle\,,\,\rangle$ on $H_{1,\mathbb{R}}$ (we do not need here the fact that the polarization is integral), which has to polarize the Hodge structure at any point, so that for any $b\in U$
\begin{eqnarray}\label{eqHR}\langle\,,\,\rangle=0\,\,{\rm on}\,\, \mathcal{H}_{1,0,b}\subset H_{1,\mathbb{C},b}\\
\nonumber
i\langle\alpha,\overline{\alpha}\rangle >0,\,\,\,\,\forall 0\not=\alpha\in\,\,
\mathcal{H}_{1,0,b}.
\end{eqnarray}
 This can be viewed as a nondegenerate skew-symmetric form $\omega\in \bigwedge^2H^1(A_{b_0},\mathbb{R})$ which
produces  via
the diffeomorphism (\ref{eqlocalident}) a closed $2$-form $\omega^*$ on the open set $U$ of the trivialization introduced above. The $2$-form $\omega^*$ is constant in the affine coordinates introduced above.
We have  the following lemma whose proof is a formal computation (see \cite{syz}):
\begin{lemm} \label{lekahler}
The Hodge-Riemann bilinear relations (\ref{eqHR}) satisfied by
$\langle\,,\,\rangle$ and the Hodge structure on $H_1(A_{b},\mathbb{R})$ at any point  $b\in U$ are equivalent to the fact that
$\omega^*$ is a K\"ahler form on $U$.
\end{lemm}

As a corollary of the above descrition, we get the following result due to Donagi and Markman \cite{donagimarkaman}:
\begin{theo}\label{theodonagi} (Donagi-Markman) Let $X\rightarrow B$ be a Lagrangian fibration. Then locally on $B_{reg}$, there exist a holomorphic function
$g$ and holomorphic coordinates $z_1,\ldots, z_n$ such that
the infinitesimal variation of Hodge structures
$T_{B,b}\rightarrow {\rm Hom}\,(H_{1,0,b},H_{0,1,b})$ at $b\in B_{reg}$ is induced by the cubic form
$g^{(3)}$ defined by the third partial derivatives of $g$, using the identifications
$T_{B,b}\cong H_{1,0,b}\cong H_{0,1,b}^*$.
\end{theo}
\begin{proof} Indeed, let $f_i$ be the local holomorphic functions appearing above. The indices correspond to a
local flat  basis $e_1,\ldots, e_{2n}$
of $H_{1,\mathbb{Z}}$. We can assume that
the basis is chosen so that $\langle e_1,\ldots,e_n\rangle$ is totally isotropic for the pairing giving the polarization.
We choose for coordinates $z_1=f_1,\ldots,z_n=f_n$.
We have
$$df_{n+i}=\sum_{j=1}^ng_{ij} dz_j,$$
with $g_{ij}=\frac{\partial f_{n+i}}{\partial z_j}$,
so that the kernel of the evaluation map
$$\mathbb{C}^{2n}\otimes \mathcal{O}_B\rightarrow \Omega_B$$
is generated by
$e_{n+i}-\sum_{j=1}^ng_{ij} e_j$ for $j=1,\ldots,n$.
As the subspace generated by these elements is totally isotropic for the pairing, we get
that $g_{ij}=g_{ji}$. Thus we have $\frac{\partial f_{n+i}}{\partial z_j}=\frac{\partial f_{n+j}}{\partial z_i}$ for any
$i,\,j$, and
this implies that locally there is a holomorphic function $g$ such that
$f_{n+i}=\frac{\partial g}{\partial z_i}$.

We finally compute the infinitesimal variation of Hodge structures by differentiating the generators
$e'_i:=e_{n+i}-\sum_{j=1}^ng_{ij} e_j$ of $\mathcal{H}_{1,0}$: we get
$$ \frac{\partial e'_i}{\partial z_k}=-\sum_{j}\frac{\partial g_{ij}} {\partial z_k} e_j=
-\sum_{j}\frac{\partial^3 g} {\partial z_i\partial z_j\partial z_k} e_j.$$
This proves the claim as the basis $e'_i,\,i\geq n+1$ is dual to the basis $e_i$ for $i\leq n$, which itself corresponds to
the basis $\frac{\partial}{\partial z_i}$ of $T_{B,b}$.
\end{proof}
\subsection{Proof of Theorem \ref{theomain} when ${\rm dim}\,X\leq 8$ and the variation is maximal}
Let $X$ be a very general hyper-K\"ahler manifold with Lagrangian fibration
$\pi: X\rightarrow B$ and $M\in{\rm Pic}\,X$ a divisor  topologically trivial on fibers of $\pi$. Assume $\pi$ is not locally isotrivial.
Let us prove that the assumptions of Theorem \ref{theoandre} are satisfied. This follows from
\begin{lemm} \label{lenovhs} Let $\phi: X\rightarrow B$ be a
 Lagrangian fibration on a hyper-K\"ahler manifold.
 Then there is no proper nontrivial real  subvariation
 of Hodge structure of $H_{1,\mathbb{R}}$ on $B^0$.
 \end{lemm}
 \begin{proof}
 According to Matsushita's Proposition \ref{promatsu}, the
restriction map
$H^2(X,\mathbb{Q})\rightarrow H^2(X_b,\mathbb{Q})$
has rank $1$. By Deligne's invariant cycle theorem, this implies that the
local system $R^2\phi_*\mathbb{R}$ on $B^0$ has only one global section (provided by the
polarization $\omega$ and its real multiples).
If the local system $(R^1\phi_*\mathbb{R})^*$ contains a nontrivial local subsystem
$L$
which carries a real variation of Hodge structures, the restriction of $\omega$
to $L$ is nondegenerate and the orthogonal complement
$L^{\perp \omega}$ with respect to $\omega$ is also a local subsystem of $(R^1\phi_*\mathbb{R})^*$ on which $\omega$ is nondegenerate. Thus
we have a decomposition
\begin{eqnarray}
\label{ledecomp}(R^1\phi_*\mathbb{R})^*=L \oplus L^{\perp }
\end{eqnarray}
and if both $L$ and $L^\perp$ were nontrivial, we would get two independent sections of $R^2\phi_*\mathbb{R}=\bigwedge^2((R^1\phi_*\mathbb{R})^{*})^*$,
namely $\pi_L^*\omega_{\mid L}$ and $\pi_{L^{\perp}}^*(\omega_{\mid L^{\perp}})$, where
$\pi_L$ and $\pi_{L^{\perp}}$ are the two projections associated with
the decomposition
(\ref{ledecomp}). This is a contradiction, hence a proper nontrivial such  $L$ does not exist.
\end{proof}
If $\nu_M$ is not of torsion, the Zariski closure of $\mathbb{Z}\nu_M(B)$ has an irreducible component which
is a subfamily of abelian varieties (of positive dimension) of ${\rm Pic}^0(X_{reg}/B_{reg})$. Such a family must be equal
to ${\rm Pic}^0(X_{reg}/B_{reg})$ by Lemma \ref{lenovhs}. This establishes assumption (ii) of Theorem
\ref{theoandre}. For the same reason, ${\rm Pic}^0(X_{reg}/B_{reg})$ has no
locally trivial subfamily because  we are in the nonisotrivial case and this establishes
assumption  (i) of Theorem
\ref{theoandre}.

\begin{proof}[Proof of Theorem \ref{theomain} when ${\rm dim}\,X\leq 8$ in  the case of maximal variation]
The maximal variation assumption tells us that
at a general point $b\in B_{reg}$
the map
$$d\mathcal{P}: T_{B,b}\rightarrow {\rm Hom}(H_{1,0,b}, H_{0,1,b})$$
is injective.

Reasoning by contradiction, if $\nu_M$ has the property that its torsion points are not dense in $B_{reg}$, then as explained above,
Theorem \ref{theoandre} applies and tells that the map
$\overline{\nabla}: H_{1,0,b}\rightarrow {\rm Hom}(T_{B,b}, H_{0,1,b})$ associated to $d\mathcal{P}$ above has the property  that
$\overline{\nabla}_\lambda\in {\rm Hom}(T_{B,b}, H_{0,1,b})$ is not an isomorphism for any $\lambda\in H_{1,0,b}$.
According to Theorem \ref{theodonagi}, these two maps are in fact identical and
the $\overline{\nabla}_\lambda$'s are  the partial derivatives
 $\partial_\lambda g^{(3)}$ of a cubic form  $g^{(3)}$ on $T_{B,b}=H_{1,0,b}$.
The homogeneous polynomials of degree $3$ in $n$ variables satisfying the condition that all of their partial derivatives are degenerate quadrics
have been classified by   Hesse and Gordan-Noether if $n\leq 4$, and by Lossen \cite{lossen} if $n\leq 5$.
We refer to section \ref{secn5} for the latter.
In the case where $n\leq4$, one has the following
\begin{theo} (see \cite{lossen} Any cubic homogeneous polynomial in $n\leq4$ variables all of whose partial derivatives are degenerate is a cone.
\end{theo}
Being a cone means that one partial derivative derivative is identically $0$, which contradicts
 the injectivity of the map $T_{B,b}\rightarrow {\rm Hom}(H_{1,0,b}, H_{0,1,b})$.
\end{proof}
\begin{rema}{\rm Part of the arguments described in this section also appear in Lin's thesis \cite{linthesis}.
}
\end{rema}

\section{Proof of Theorem \ref{theomain} when the fibration is isotrivial \label{sec3}}
\subsection{Further general facts about sections of lagrangian fibrations\label{secpetitechose}}
Let $\phi: X\rightarrow B$ be a Lagrangian fibration, where $X$ is a projective irreducible hyper-K\"ahler
manifold of dimension
$2n$,
and let $B^0\subset X^0$ be the Zariski open set of regular values of $\phi$. Denote by $\sigma$  the holomorphic $2$-form on $X$ and by $\omega$  the first Chern class of an ample line bundle $H$ on $X$.
Over $B^0$, the fibration $\phi:X^0\rightarrow B^0$ is a fibration into abelian varieties
or rather a torsor on the corresponding Albanese fibration
${\rm Alb}(X^0/B^0)$, which  admits a multisection $Z\rightarrow B^0$ of degree $d$.
Using these data, $X^0\rightarrow B^0$
 is isogenous  to the associated Albanese fibration ${\rm Alb}(X^0/B^0)$ via the map
 $$i_Z: X^0\rightarrow {\rm Alb}(X^0/B^0)$$
 defined by $i_Z(u)={\rm alb}_{X_b}(d\{u\}-Z_b),\,\,b=\phi(u)$. Here  the notation
 $\{u\}$ is used to denote $u$ as a $0$-cycle
 in the fiber $X_b$ (in order to avoid
 the confusion between addition of cycles and addition of points), so
 $d\{u\}-Z_b$ should be thought as a $0$-cycle of degree
$0$ on $X_b$, giving rise to an element ${\rm alb}_{X_b}(d\{u\}-Z_b)\in{\rm Alb}(X_b)$.
Next, using the polarisation $H$, the abelian fibration
${\rm Alb}(X^0/B^0)$ and  its dual fibration ${\rm Pic}^0(X^0/B^0)$ are isogenous,
  the isogeny
\begin{eqnarray}\label{eqisog} i_H:{\rm Alb}(X^0/B^0)\rightarrow {\rm Pic}^0(X^0/B^0)
\end{eqnarray}
being given by $$i_H(b, v_b)=t_{v_b}^*(H_{\mid X_b})-H_{\mid X_b}$$
for any $b\in B^0$, $v_b\in {\rm Alb}(X_b)$. Here
the translation $t_{v_b}$ acts on $X_b$.
 The cycle $H^{n-1}\in{\rm CH}^{n-1}(X)$ also gives  an isogeny
\begin{eqnarray}\label{eqisogdual} {\rm alb}_{X^0/B^0}\circ H^{n-1}:{\rm Pic}^0(X^0/B^0)\rightarrow
{\rm Alb}(X^0/B^0)
\end{eqnarray}
which is easily shown to provides an inverse of (\ref{eqisog}) up to a nonzero coefficient.

Note now that the $(2,0)$-form $\sigma$ on $X$ is the pull-back of a $(2,0)$-form $\sigma_A$ on ${\rm Alb}(X^0/B^0)$ (or on ${\rm Pic}^0(X^0/B^0)$)
via the rational map $i_Z:X\dashrightarrow {\rm Alb}(X^0/B^0)$. This indeed follows easily from
Mumford's theorem \cite{mumford}, since for any two points $x,\,y$ of
$X^0$ such that $i_Z(x)=i_Z(y)$ the difference ${\rm alb}_{X_b}(\{x\}-\{y\})$ is a torsion point in
${\rm Alb}\,(X_b)$, $b=\phi(x)=\phi(y)$, hence the cycle $\{x\}-\{y\}$ is a torsion cycle
in ${\rm CH}_0(X_b)$ and a fortiori in ${\rm CH}_0(X)$ (in fact it actually vanishes in ${\rm CH}_0(X)$
since the later group has no torsion).
Using the maps above,  we get   as well  a  nondegenerate $(2,0)$-form $\sigma_A$ on
$A:={\rm Pic}^0(X^0/B^0)$, extending to a smooth projective completion, and for which the fibration is Lagrangian. Furthermore $\sigma=((i_H\circ i_Z)^*\sigma_A$.

Our analysis of the  properties of the holomorphic section $\nu$ of
the torus fibration will be infinitesimal hence local, but we want first to add one restriction
which is global:
The Lagrangian fibration $\phi:X\rightarrow B$ being as above, let
 $M\in{\rm Pic}\,X$  be a line bundle on
$X$ which is topologically trivial on the fibers and let $\nu=\nu_M:B^0\rightarrow A $ be the
associated  normal function.
\begin{lemm} \label{lelag} The section $\nu$ is Lagrangian for the holomorphic
$2$-form $\sigma_A$.
\end{lemm}
\begin{proof} Let $\widetilde{B}$ be a desingularization of $B$. Then $H^0(\widetilde{B},\Omega_{\widetilde{B}}^2)=0$ since a nonzero holomorphic $2$-form on $\widetilde{B}$
would provide by pull-back via the rational map
$\phi:X\dashrightarrow \widetilde{B}$ a nonzero holomorphic form on $X$ which is not
proportional to $\sigma$. So we just have to prove that
$\nu^*(\sigma_A)$ (which is a priori only defined as a holomorphic $2$-form
on $B^0$) extends to a holomorphic $2$-form on $\widetilde{B}$. However, we can provide another
construction of $\nu^*(\sigma_A)$ which will make clear that it extends to
$\widetilde{B}$: Looking at the definitions of $i_H,\,i_Z$, the subvariety
$\Gamma:=(i_H\circ i_Z)^{-1}({\rm Im}\,\nu)$ is algebraic in $X^0$ and it is a multisection
of $\phi$  over $B^0$. Its Zariski closure in $\widetilde{B}\times X$ will provide a
correspondence $\overline{\Gamma}\subset \widetilde{B}\times X$.
Then as $\sigma=i_Z^*\sigma_A $ on $X^0$, we get that
\begin{eqnarray}\label{eqform} \overline{\Gamma}^*\sigma=N\nu^*(\sigma_A)
\end{eqnarray}
on $B^0$, where $N$ is the degree of $i_H\circ i_Z$. This gives the desired extension of
$\nu^*(\sigma_A)$.
\end{proof}

\subsection{Degenerate normal function and degenerate Monge-Amp\`ere equation\label{secMA}}
We complete the Donagi-Markman analysis in section \ref{subsec}
 by adding to the data described there the normal function $\nu:U\rightarrow A$, which
satisfies the property that $\nu^*\sigma_A=0$, which is the global restriction coming from Lemma \ref{lelag}.
This normal function $\nu\in \Gamma(U,\mathcal{A})=\Gamma(U,\mathcal{H}_{0,1}/H_{1,\mathbb{Z}})$ lifts locally to  a section
$\tilde{\nu}$
of $\mathcal{H}_{0,1}\cong \Omega_U$, hence provides locally a holomorphic $1$-form
$\alpha$ on $U$.
\begin{lemm}\label{leclosed}
The form $\alpha$ is closed.
\end{lemm}
\begin{proof}
It is a general fact of the theory of  integrable systems (see \cite{donagimarkaman})
that the isomorphism (\ref{eqisovbundlesdual}) is also constructed in such a way that the
pull-back $\tilde{\sigma}_A$ of the holomorphic  $2$-form $\sigma_A$ to the total space of the bundle
$\mathcal{H}_{0,1}$ identifies with the canonical $2$-form
$\sigma_{can}$
on the cotangent bundle $\Omega_U$, which has the property that
for any  holomorphic $1$-form $\beta$ on $U$, seen as a section
of $\Omega_U$, one has
\begin{eqnarray}
\label{eqsigmacan}
d\beta={\beta}^*(\sigma_{can}).\end{eqnarray}
As the section
$\tilde{\nu}$ corresponds via this isomorphism to the $1$-form
$\alpha$, the assumption that
${\nu}^*({\sigma}_{A})=0$ says  equivalently that
$$0=\tilde{\nu}^*(\tilde{\sigma}_{A})=\alpha^*(\sigma_{can})=d\alpha.$$
\end{proof}
The closed holomorphic $1$-form $\alpha$ thus provides us on simply connected open sets $U'$ of $U$ with a holomorphic function
$f$  such that $df=\alpha$.
\begin{rema}{\rm From now on, the holomorphic $2$-form on our Lagrangian fibration will disappear. It is hidden in the fact that we work with  the cotangent bundle of the base, which has a canonical holomorphic $2$-form.
}
\end{rema}

We are in the local setting where the holomorphic functions
 $f_i,\,f$ constructed in the previous section exist.
Let thus $V$ be a complex manifold, and let
$f_1,\ldots, f_{2n}$ be holomorphic functions on $V$ such that
the functions $x_i:={\rm Re}\,f_i$ provide real coordinates on $V$, giving rise
to a real variation of Hodge structures
\begin{eqnarray}\label{eqrealvhs}\mathbb{R}^{2n}\otimes_\mathbb{R} \mathcal{O}_V\stackrel{ev}{\rightarrow} \Omega_V\rightarrow 0,\\
\nonumber e_i\mapsto df_i
\end{eqnarray}
 with underlying local system $\mathbb{R}^{2n}=T_{V,\mathbb{R}}$ and Hodge bundles
 $\mathcal{H}_{1,0}={\rm Ker}\,ev,\,\mathcal{H}_{0,1}=\Omega_V$.
Let $f$ be  a holomorphic function on $V$; as the $df_i$'s are independent over $\mathbb{R}$ in
$\Omega_{V,b}$ at any point $b\in V$, we can write the equality
\begin{eqnarray}\label{eqpourdf1}df=\sum_i a_i df_i,\end{eqnarray}
of $\mathcal{C}^\omega$ forms of type $(1,0)$ on $V^0$,
where the $a_i$'s are real $\mathcal{C}^\omega$ functions on $V$, uniquely determined by (\ref{eqpourdf1}).
 Denoting $x_i:={\rm Re}\,f_i$, the $x_i$'s thus provide locally
real coordinates on $V^0$.
 Note that (\ref{eqpourdf1}) is obviously equivalent, since the $a_i$'s are real and the considered
forms are of type $(1,0)$, to the equality of real $1$-forms:
\begin{eqnarray}
\label{eqpourdfreel}
d({\rm Re}\,f)=\sum_i a_i dx_i.
\end{eqnarray}
We assume that the map $a_\cdot=(a_1,\ldots,a_{2n}):V^0\rightarrow\mathbb{R}^{2n}$
is of constant rank $2n-2k$ (the rank is even by Lemma \ref{leker}). The condition that $a_\cdot$ is not of maximal rank at the general point is a degenerate real  Monge-Amp\`ere equation
(see \cite{refMA}).
\begin{lemm}\label{leaffine} The fibers of $a_\cdot$ are affine in the coordinates $x_i$.
\end{lemm}
\begin{proof} As the map $a_\cdot$ has locally constant rank  near $b\in V^0$, it factors through a
submersion $V^0\rightarrow \Sigma$ and an immersion $\Sigma\hookrightarrow \mathbb{R}^{2n}$, where
$\Sigma$ is real manifold of dimension $2n-2k$.
Equation (\ref{eqpourdfreel}) says equivalently that for fixed $\sigma=(\lambda_1,\ldots,\lambda_{2n})\in \Sigma$,
the function ${\rm Re}\,f-\sum_i\lambda_i x_i$ has vanishing differential on $V^0$ at any point of the fiber
$V^0_\sigma=a_\cdot^{-1}( \lambda_1,\ldots,\lambda_{2n})$.  The function
${\rm Re}\,f-\sum_ia_i x_i$ on $V^0$ is thus the pull-back $a_\cdot^* \psi$ of a function
$\psi$
on $\Sigma$, that is
\begin{eqnarray}
\label{eqpouraffinele} {\rm Re}\,f-\sum_ia_i x_i-a_\cdot^*\psi=0
\end{eqnarray}
on $V^0$.
If $u=(u_1,\ldots,u_{2n})\in T_{\Sigma,\sigma}$, we get  for any $b\in V^0_\sigma$, by differentiating (\ref{eqpouraffinele}) along any
$v\in T_{V^0,b}$ such that
${a_\cdot}_*(v)=u$ and applying (\ref{eqpourdfreel})
$$0=-\sum_iu_i x_i-d\psi(u),$$ which says that the function
$\sum_iu_i x_i$ is constant
 along the fiber $V^0_\sigma$. This provides $2n-2k$  real  equations, affine in the $x_i$'s, vanishing
 along $V^0_\sigma$ and independent over $\mathbb{R}$, and they must  define $V^0_\sigma$ in $V^0$ since $V^0_\sigma$ is a real submanifold
  of $V$ of codimension $2n-2k$.
\end{proof}
The following lemma was partially proved in the proof of Lemma \ref{leker}.
\begin{lemm} \label{lefibrecomp} The fiber $V^0_\sigma$ is a complex submanifold of $V^0$, along
which the functions $\sum_iu_i f_i$ are constant, for any $u_\cdot\in T_{\Sigma,\sigma}$.
\end{lemm}
\begin{proof} The fact that the functions $\sum_iu_i f_i$ are constant along $V^0_\sigma$ for any $u_\cdot\in T_{\Sigma,\sigma}$
is proved exactly as the previous lemma, replacing (\ref{eqpourdfreel}) by (\ref{eqpourdf1}). This implies that $V^0_\sigma$ is a complex submanifold of $V^0$ since
the real parts of these functions define $V^0_\sigma$ as a real submanifold of $V^0$ by Lemma \ref{leaffine}.
\end{proof}

 Theorem \ref{theomain} when the fibration is isotrivial works in any dimension.
 It will be obtained
as a consequence of the following Proposition \ref{prosubstitute}.
Let $\phi:A\rightarrow B$ be a locally isotrivial Lagrangian torus fibration,
where $A,\,B$ are quasiprojective and
$A$  has a generically nondegenerate closed holomorphic $2$-form
$\sigma_A$ extending to a projective completion $\overline{A}$ and generating
$H^{2,0}(\overline{A})$. Let $\nu: B\rightarrow A$ be an algebraic section such that $\nu^*\sigma_A=0$.

\begin{prop} \label{prosubstitute}
If the torsion points of $\nu$ are not dense in $B^0$ for the usual topology, either the section
$\nu$ is locally constant, or
the local system
$$H_{1,\mathbb{R},A}:= (R^1\phi_*\mathbb{R})^*$$
admits a proper nontrivial local subsystem $L$ which underlies a real subvariation of Hodge structures.
Furthermore, writing the real lift (see Section \ref{sec1}) of the section $\nu$ as a sum $\tilde{\nu}_{\mathbb{R},L}+\tilde{\nu}_{\mathbb{R},L^{\perp}}$, where
$L^{\perp}$ is the orthogonal variation of Hodge structures with respect to $\omega$, then
$\tilde{\nu}_{\mathbb{R},L}$ is a locally constant section of $L$.
\end{prop}

\begin{proof}
The local data described in Section \ref{subsec}
take the following form: there are
locally $2n$ holomorphic functions $f_1,\ldots, f_{2n}$ on simply connected open subsets $U\subset B$, which are independent over $\mathbb{R}$ and provide real coordinates $x_1,\ldots, x_{2n}$, $x_i={\rm Re}\,f_i$. By isotriviality,
  modulo the constants, the $f_i$'s provide only
$n$ independent holomorphic functions independent over $\mathbb{C}$.
 In particular  the open set $B^0$ over which $\phi$ is smooth and $\sigma_A$ is nondegenerate is endowed with a flat holomorphic structure.
\begin{lemm}\label{leflatholo}
This flat holomorphic structure, after passing to a generically finite cover
$B'$ of  $B^0$, is induced by the choice of $n$ holomorphic $1$-forms on
 a smooth projective completion $\overline{B'}$  of $B'$.
\end{lemm}
\begin{proof} Indeed, the fibration $\phi:A\rightarrow B$ is trivialized over
a finite cover
$B'$ of a Zariski open set of $B^0$ as
$$A':=A\times_{B}B'\cong B'\times J_0,$$
where $0\in B'$ is a given point, and $J_0$ is the fiber of $\phi':A'\rightarrow B'$ over it.
The local flat holomorphic coordinates on $B^0$, pulled back to $B'$, come
from the holomorphic $2$-form on $B'\times J_0$ which is of the form
$$\sum_i pr_{B'}^*\alpha_i\otimes pr_{J_0}^*\beta_i,$$
for a basis $\beta_i$ of $H^0(\Omega_{J_0})$. The flat structure on $B^0$ is given by
the holomorphic forms $\alpha_i$.
We now claim that the forms $\alpha_i$ extend  to
holomorphic $1$-forms on $\overline{B'}$. This is seen as follows: As the fibration is Lagragian and
$\sigma_A$ generates $H^{2,0}(\overline{A})$, the transcendental cohomology
$H^2(\overline{A},\mathbb{Q})_{tr}$, pulled-back to $B'\times J_0$, falls in the
$(1,1)$-K\"unneth component
$H^1(B',\mathbb{Q})\otimes H^1(J_0,\mathbb{Q})$ of the K\"unneth decomposition of
$H^2(B'\times J_0,\mathbb{Q})$, and this implies that it is contained in the
image of $H^1(\overline{B'},\mathbb{Q})\otimes H^1(J_0,\mathbb{Q})$
in $H^1(B',\mathbb{Q})\otimes H^1(J_0,\mathbb{Q})$ by Deligne's strictness theorem for
morphisms of mixed Hodge structures \cite{deligne}): indeed,
we get from the above a morphism of mixed Hodge structures
$$ \alpha: H^2(\overline{A},\mathbb{Q})_{tr}\otimes H^1(J_0,\mathbb{Q})^*\rightarrow H^1(B',\mathbb{Q})$$
with image $K\subset H^1(B',\mathbb{Q})$ such that
$H^2(\overline{A},\mathbb{Q})_{tr}$, pulled-back to $B'\times J_0$ is contained in
$K\otimes H^1(J_0,\mathbb{Q})$. By Deligne strictness theorem \cite{deligne}, this morphism $\alpha$ has its
image contained in the pure (smallest weight) part of $H^1(B',\mathbb{Q})$, namely
$H^1(\overline{B'},\mathbb{Q})$.
This finishes the proof.
\end{proof}
As a consequence, we get the following:
\begin{coro} \label{coroconstant} Let $(F_t)_{t\in \Sigma}$ be a continuous family of
complex submanifolds of an open set $V^0$ of $B^0$.
Assume that each $F_t$ is flat for the flat holomorphic structure described above, and is
algebraic, that is, a connected component of the intersection of
an algebraic subvariety $C_t$ of $B$ with $V^0$.
Then the $F_t$'s are parallel, that is, $F_t$ being affine in the flat coordinates,
the linear part of their defining equations is constant.

\end{coro}
\begin{proof} The linear part of the defining equations for $F_t$ is determined, using Lemma
\ref{leflatholo}, by the vector space of  holomorphic forms on $\overline{B'}$ vanishing on
$F_t$, or equivalently on $C_t$. On the other hand, by desingularization and
up to shrinking our family, we can assume that the $C_t$ form a family of smooth
projective varieties mapping to $\overline{B'}$. It is then clear that the space
of holomorphic forms on $\overline{B'}$ vanishing
 on $C_t$ does not depend on $t$.
\end{proof}

The proof of Proposition  \ref{prosubstitute}   will finally use the following Lemma \ref{lealgebraic30sept}.
Consider our locally isotrivial fibration
$\phi:A\rightarrow B$ and its  section
$\nu$, which  satisfies the condition $\tilde{\nu}^*\sigma_A=0$.
This section
lifts locally on $B^0$ to a   section $\tilde{\nu}_M$ of the sheaf
$H_{1,\mathbb{R}}\otimes \mathcal{C}^\infty_\mathbb{R}$ where
$H_{1,\mathbb{R}}$ is the
the real local system $(R^1\phi_*\mathbb{R})^*$.
This section is well-defined up to translation by a constant (integral) section of
$H_{1,\mathbb{R}}$, and by Proposition \ref{ledense}, if the torsion points of
$\nu_M$ are not dense in $B^0$ for the usual topology, then
the map (defined on simply connected open sets $V\subset B^0$)
$$\tilde{\nu}_M: V\rightarrow H_1(J_0,\mathbb{R})$$
via a trivialization of $H_{1,\mathbb{R}}$ has positive dimensional fibers
$F_t$.
\begin{lemm}\label{lealgebraic30sept} In the isotrivial case, the fibers $F_t$ are
 algebraic.
\end{lemm}

\begin{proof} Indeed, after passing to the finite cover $B'$ already introduced above,
the normal function $\nu_M$ gives a morphism of algebraic varieties
$$\nu'_M:{B'}\rightarrow {\rm Pic}^0(J_0)\cong_{isog} J_0,$$
which to $b'\in B'$ associates $M_{\mid X'_{b'}}$ with $X'_{b'}\cong J_0$ canonically.
It is imediate that the fibers of $\tilde{\nu}_M$ coincide with the fibers of $\nu'_M$, via the
morphism $B'\rightarrow B^0$ and after restricting to the adequate open sets.
\end{proof}

We can now apply Corollary \ref{coroconstant} to the fibers $V^0_\sigma\subset V^0$ of the local map $a_\cdot$ restricted to the open subset $V^0\subset V$ where $a_\cdot$ has constant rank. These fibers
 are both complex and affine by Lemmas \ref{leaffine} and \ref{lefibrecomp}.
We thus conclude that they vary in a parallel way. The tangent bundle to the fibers is thus a local subsystem of
$H_{1,\mathbb{R}}$, and as the fibers are complex submanifolds, this local system is a subvariation of Hodge structure.
This subvariation of Hodge structures is nontrivial by assumption. Saying that it is equal to
$H_{1,\mathbb{R}}$ is equivalent to saying that $\nu$ is locally constant. This finishes the proof.
\end{proof}
\begin{proof}[Proof of Theorem \ref{theomain} in the isotrivial case] Let $\phi:X\rightarrow B$ be a Lagrangian isotrivial fibration, $X$  projective hyper-K\"ahler, and let $\nu_M$ be the normal function associated to a line bundle $M$ on $X$ which is topologically trivial on the smooth fibers
of $\phi$. Proposition \ref{prosubstitute}
applies and shows that if the set of points $b\in B$ where $\nu_M(b)$ is a torsion point is not dense in $B$,
then either the normal function is locally constant or there is a proper nontrivial real subvariation of Hodge structure
$L$
of $H_{1,\mathbb{R}}$. We now use
 Lemma \ref{lenovhs} which says
 that the second  case is excluded and
thus $\tilde{\nu}$ is locally constant, so that $\nu_M$ is a section
of $R^1\phi_*\mathbb{R}/R^1\phi_*\mathbb{Z}$. But then Lemma \ref{letorsionmono} says that
a nontorsion section $R^1\phi_*\mathbb{R}/R^1\phi_*\mathbb{Z}$ can exist only if the
invariant part of $H^1(X_b,\mathbb{R})$ is not trivial, which provides a contradiction by Deligne's global invariant cycles theorem since $H^1(X,\mathbb{R})=0$.
\end{proof}

\section{The case of dimension $10$ \label{secn5}: proof of Proposition \ref{prodim10}}
We prove in this section  Proposition \ref{prodim10}. We start with the proof of (i).
We consider a projective hyper-K\"ahler $10$-fold $X$   admitting a Lagrangian fibration $\pi:X\rightarrow B$ with maximal variation. Let $M\in {\rm Pic}\,X$ be cohomologous to $0$ on the fibers $X_b$.  We assume that the associated normal function $\nu_M:B^0\rightarrow {\rm Pic}\,X^0/B^0$ does not have  a dense set of torsion points in $B^0$. We use the induced Lagrangian fibration structure on the dual fibration ${\rm Pic}\,(X^0/B^0)\rightarrow B^0$ discussed in Subsection
\ref{secpetitechose}, that we denote by $A\rightarrow B^0$.
We perform the same analysis as in Section \ref{sec2}. By Theorem
\ref{theodonagi}, there are locally on $B^0$ holomorphic coordinates
 $z_1,\ldots, z_5$ and a holomorphic function (a potential)
$g$ with the following properties:
Let $g_i=\frac{\partial g}{\partial z_i}$. Then the
holomorphic forms $dz_1,\ldots,\,dz_5,\,dg_1,\ldots,\,dg_5$ are pointwise independent over $\mathbb{R}$ and
the real variation of Hodge structure on $H_{1,A}$ is given by the evaluation map
$${\rm ev}:\mathbb{R}^{10}\otimes \mathcal{O}_B\rightarrow \Omega_B,\,\,e_i\mapsto dz_i,\,\,e_{i+5}\mapsto dg_i,\,\,i=1,\ldots,\,5.$$
Thus the infinitesimal variation of Hodge structure
$$T_{B}\rightarrow {\rm Hom}\,(\mathcal{H}_{1,0},\mathcal{H}_{0,1})$$
is given, using the natural identifications
$T_B\cong \mathcal{H}_{0,1}^*\cong  \mathcal{H}_{1,0}$,
 by the homogeneous cubic form $g^{(3)}$ on $T_B$ determined  at each point by the partial derivatives
 $\frac{\partial ^3g}{\partial z_i\partial z_j\partial z_k}$.
 The maximal variation assumption says that the cubic form $g^{(3)}_b$ is not a cone at a general point $b$,
 while the non-density assumption for the torsion points implies by the Andr\'e-Corvaja-Zannier theorem \ref{theoandre} that
 it has the property that all of its partial derivatives are degenerate quadrics.
 For cubic forms in $5$ variables, such cubics are classified in \cite{lossen}, who proves that such a cubic form on $\mathbb{P}^4$ has to be singular along a plane
 (and  conversely, if  a cubic in $\mathbb{P}^4$  is singular along a plane, then  all of its partial derivatives  vanish along a plane, hence are degenerate quadrics).
 We now prove the following result:
 \begin{lemm}\label{ledecompaffine}  Let $g(z_1,\ldots,z_5)$ be a holomorphic function
  defined on an open set
  $U$ of $\mathbb{C}^5$ with the property that
  at any point $b\in U$, the cubic form $g^{(3)}_b$ is not a cone and is singular along a
  projective plane $P_b\subset \mathbb{P}(T_{U,b})$. Then
  there is locally on $U$ a holomorphic map $\phi:U\rightarrow \mathbb{C}^2$
  with  fibers  $S_b$ affine in the coordinates $z_i$, such that the partial derivatives
  $\frac{\partial g}{\partial z_i}$ are affine on each $S_b$. In particular,
  $g$ restricted to the fibers $S_b$ is quadratic.
  \end{lemm}
 \begin{proof} Let us consider the Taylor expansion of $g$ in the coordinates $z_i$ centered at $b\in B$:
 \begin{eqnarray}\label{eqtayor} g=q_b+g^{(3)}_b+g^{(4)}_b+...,
 \end{eqnarray}
 where $q_b$ is quadratic nonhomogeneous in the coordinates, and  the homogeneous forms $g^{(k)}_b$ in the variables $X_i=z_i-b$ depend polynomially on $b$ for $k\geq 3$. We observe that
 \begin{eqnarray}\label{eqtayor1}\frac{\partial g^{(3)}_b}{\partial z_i}=\pm \frac{ \partial g^{(4)}_b}{\partial X_i}.
 \end{eqnarray}
By assumption, the cubic form $g^{(3)}_b$ is singular along a projective plane
$P_b\subset \mathbb{P}^4$ which is uniquely determined by $b$ hence varies holomorphically with $b$ (indeed, otherwise the cubic form $g^{(3)}_b$  is a cone, as one checks easily).
It follows that $\frac{\partial g^{(3)}_b}{\partial z_i}$ vanishes on $P_b$, which means
by (\ref{eqtayor1}) that the cubic $\frac{\partial{g^{(4)}_b}}{\partial X_i}$ vanishes on $P_b$. Hence the quartic $g^{(4)}_b$ is also singular along $P_b$. We can then continue and show inductively that
all $g^{(k)}_b$ for $k\geq3$ are singular along $P_b$.
In conclusion, fixing $b$, we have by (\ref{eqtayor})
\begin{eqnarray}\label{eqtayor2}g=q_b+g'_b,
 \end{eqnarray}
where $q_b$ is  quadratic and  $g'_b $ vanishes doubly along the affine $3$-space  $S_b:=\mathbb{C}^3\cap U\subset U$ corresponding to $P_b$.
This is saying that the map  $b\mapsto P_b$ is constant along $S_b$, which provides us with a foliation  on $U$ with affine fibers $S_b$, with the property that for each $b$,
(\ref{eqtayor2}) holds along $S_b$ with $q_b$ constant along $S_b$. Differentiating (\ref{eqtayor2}) along
$S_b$ gives $\frac{\partial g}{\partial z_i}=\frac{\partial q_b}{\partial z_i}$ along $S_b$ which proves
that $\frac{\partial g}{\partial z_i}$ is affine on $S_b$.
\end{proof}
 Lemma \ref{ledecompaffine} tells us that for our data $(X,B,d)$ as above,
$B^0$ (or rather the open set where the cubic is not a cone) is foliated by the locally defined fibers $S_b$, which clearly are both
complex submanifolds and affine submanifolds of $B^0$ for the real affine structure. Indeed, the partial derivatives $g_i:=\frac{\partial g}{\partial z_i}$ are affine along $S_b$, so that
each $S_b$ maps to an affine $\mathbb{C}^3\subset \mathbb{C}^{10}$ via the functions $z_1,\ldots,\,z_5,\,g_1,\ldots,\, g_5$, hence to a real affine subspace
$\mathbb{R}^6\subset \mathbb{R}^{10}$ via the real  coordinates
${\rm Re}\,z_1,\ldots,\,{\rm Re}\,z_5,\,g_1,\ldots, \,{\rm Re}\,g_5$.

The proof of Proposition \ref{prodim10} (i) concludes with the following :
 \begin{coro}\label{coroHSconst} Along the leaves $\mathcal{F}_b$ of the foliation,
 the real variation of Hodge structure is a direct sum
 $L_{b,\mathbb{R}}\oplus L'_{b,\mathbb{R}}$, where the real Hodge structure is locally constant on the first summand.
 \end{coro}
\begin{proof} By Lemma \ref{ledecompaffine}, the leaf $\mathcal{F}_b$, which identifies locally to   $S_b$, locally maps, via the real affine structure given by $({\rm Re}\,z_i,\,{\rm Re}\, g_i),\,1\leq i\leq 5$, to an open set of an affine subspace
$\mathbb{R}^6\subset \mathbb{R}^{10}$ which will define $L'_{b,\mathbb{R}}\subset
H_{1,\mathbb{R}\mid \mathcal{F}_b}$ as the four dimensional subspace of real affine forms vanishing on $S_b$. The decomposition is then obtained using the polarization, $L_{b,\mathbb{R}}$ being defined as the orthogonal complement of $L'_{b,\mathbb{R}}$. We have to prove that $L'_{b,\mathbb{R}}$ and $L_{b,\mathbb{R}}$ are   real
 subvariations of Hodge structure of $H_{1,\mathbb{R}\mid \mathcal{F}_b}$, and that
the  variation of Hodge structre on  $L_{b,\mathbb{R}}$ is  constant. The first point  follows from the fact that $S_b\subset U$ is also a complex submanifold of $U$ so that  the evaluation map
 $L'_{b,\mathbb{R}}\otimes_{\mathbb{R}}\mathcal{O}_B\rightarrow \Omega_U$
 maps to holomorphic $1$-forms vanishing on $S_b$, which is a rank $2$ vector subbundle of $\Omega_{U\mid S_b}$. The fact that the real variation of Hodge structure on $L_{b,\mathbb{R}}$ is  constant is due to the fact that it identifies  to the evaluation map
 $L_{b,\mathbb{R}}\otimes_{\mathbb{R}}\mathcal{O}_{S_b}\rightarrow \Omega_{S_b},\,\,e_i\mapsto df_i$,
 associated with a $6$-dimensional real
 vector space of    functions $f_i$ on $S_b$, where the $f_i$'s form only a $3$-dimensional
  space of holomorphic  functions on $S_b$ since $S_b$  maps via the $f_i$'s  to an open subset of  an affine $\mathbb{C}^3\subset \mathbb{C}^{10}$.
\end{proof}
\begin{proof}[Proof of Proposition \ref{prodim10}, (ii)] We only have to prove that the normal function
$\nu_M$ (or rather its lift $\tilde{\nu}_{M,0,1}\in \mathcal{H}_{0,1}$) decomposes locally along the leaves,  according to the decomposition of Corollary \ref{coroHSconst},
as $\nu_{M,0,1}^v+\nu_{M,0,1}^h$, where $\nu_{M,0,1}^v$ is constant.
We use the local description introduced in  section
\ref{secMA}: the lift $\tilde{\nu}_{M,0,1}\in \mathcal{H}_{0,1}\cong \Omega_B$ provides a holomorphic $1$-form which is closed hence locally exact
$df$. The variation of Hodge structure is given in local
coordinates $z_1,\ldots,z_5$ by the evaluation map
$$\mathbb{C}^{10}\otimes \mathcal{O}_B\rightarrow \Omega_B,$$
$$e_i\mapsto dz_i,\,\,e_{5+i}\mapsto dg_i,\,i=1,\ldots,\,5,$$
where $g_i=\frac{\partial g}{\partial z_i}$ and $g$ is as above.
Writing $df=\sum_i\frac{\partial f}{\partial z_i}dz_i$, the
affine bundle
$\mathcal{H}_{0,1,\nu}\subset\mathbb{C}^{10}\otimes \mathcal{O}_B$ introduced in 
(\ref{eqh10nu}) is the set of elements $\sum_{i=1}^5\frac{\partial f}{\partial z_i} e_i
+\sum_{i=1}^5\lambda_i (e_{i+5}-\sum_{j=1}^5g_{ij} e_j)$.
 We apply now Theorem \ref{theoandrevariant} which tells us that
the holomorphic map $B\times \mathbb{C}^5\rightarrow \mathbb{C}^{10}$
\begin{eqnarray}
\label{eqmap19dec}(b,\lambda_i)\mapsto \sum_{i=1}^5\frac{\partial f}{\partial z_i} e_i
+\sum_{i=1}^5\lambda_i (e_{i+5}-\sum_{j=1}^5g_{ij} e_j)
\end{eqnarray}
is nowhere of maximal rank. Computing its differential at a point
$(b,\lambda_\cdot)$, we conclude that the  quadric
$$f^{(2)}_b-\partial_{\lambda_\cdot}( g^{(3)}_b)$$
is degenerate on $T_{B,b}$ for all $\lambda_\cdot$.
As the cubic form $g^{(3)}_b$ at the general point $b$ is not a cone but is singular along the plane
$P_b\subset\mathbb{P}(T_{B,b})$ defined by two homogeneous equations $X_1,\,X_2$, the cubic form
has for equation $X_1^2A+X_2^2B+X_1X_2C$, where the linear forms
$A,\,B,\,C$ restrict to  independent linear forms  on $P_b$, hence the  general quadric $\partial_{\lambda_\cdot}( g^{(3)}_b)$ has rank $4$ and its vertex is a general point of $P_b$. The condition that $\mu f^{(2)}_b-\partial_{\lambda_\cdot} (g^{(3)}_b)$ is degenerate for any $\mu$  implies that the Hessian $f^{(2)}_b$ of $f$ at $b$ vanishes on the vertex of the quadric defined
  by $\partial_{\lambda_\cdot} (g^{(3)}_b)$, and as this is true for general  $\lambda_\cdot$, $f^{(2)}_b$ has to vanish on  $P_b$. Thus the restricted function
$f_{\mid S_b}$ has trivial Hessian, that is,
$f$ is affine on $S_b$. Thus $d(f_{\mid S_b})$ is a constant $1$-form for the affine structure given by the $z_i$'s, which concludes the proof.
\end{proof}

Coll\`{e}ge de France, 3 rue d'Ulm, 75005 Paris FRANCE

claire.voisin@imj-prg.fr

\begin{thebibliography}{99}
\bibitem{andremono} Y. Andr\'e. Mumford-Tate groups of mixed Hodge structures and the theorem of the fixed part,  Comp. Math. 82 (1992), 1-24.
\bibitem{ACZ} Y. Andr\'e, P. Corvaja, U. Zannier. On the Betti map associated to abelian logarithms, preprint in  preliminary form  September 2017.
\bibitem{beaudo} A. Beauville, R. Donagi. La vari\'{e}t\'{e} des droites d'une hypersurface cubique de dimension $4$, {\it C. R. Acad. Sci. Paris S\'{e}r. I Math.}  {\bf 301}  (1985),  
703--706.
\bibitem{beau} A. Beauville.  On the splitting of the Bloch-Beilinson filtration, in
 {\it Algebraic cycles and motives} (vol. 2), London Math. Soc. Lecture Notes 344, 38-53; Cambridge University Press (2007).
\bibitem{beauvoi} A. Beauville, C. Voisin.
On the Chow ring of a $K3$ surface,  J. Algebraic Geom. 13 (2004),
no. 3, 417--426.
\bibitem{bover}   F. Bogomolov. On the cohomology ring of a simple hyper-K\"ahler manifold (on the results of Verbitsky). Geom. Funct. Anal. 6 (1996), no. 4, 612-618.
    \bibitem{brosnan} P. Brosnan,  G. Pearlstein. The zero locus of an admissible normal function, Annals of Mathematics 170, no.2 (2009):883-97.
    \bibitem{carlson} J. Carlson. Extensions of Mixed Hodge Structures. In {\it Journ\'{e}es de G\'{e}ometrie Alg\'{e}brique d'Angers, Juillet 1979}, 107–27. Alphen aan den Rijn: Sijthoff and Noordhoff, (1980).
    \bibitem{charles} F. Charles. On the zero locus of normal functions and the \'{e}tale Abel-Jacobi map. Int. Math. Res. Not. IMRN 2010, no. 12, 2283-2304.
    \bibitem{ciliharrismiranda} C. Ciliberto, J.  Harris, R. Miranda. General components of the Noether-Lefschetz locus and their density in the space of all surfaces. Math. Ann. 282 (1988), no. 4, 667-680.
    \bibitem{coleman} R. Coleman.
Manin's proof of the Mordell conjecture over function fields.
Enseign. Math. (2) 36 (1990), no. 3-4, 393-427.
\bibitem{collino} A. Collino. Indecomposable motivic cohomology classes on quartic surfaces and on cubic fourfolds, in {\it  Algebraic K-theory and its applications} (Trieste, 1997), 370-402, World Sci. Publ., River Edge, NJ, 1999.
    \bibitem{zannier} P. Corvaja, D. Masser, U. Zannier. Torsion hypersurfaces on abelian schemes and Betti coordinates,
   preprint 2016,  to appear in Math. Annalen.
    \bibitem{deligne} P. Deligne. Th\'eorie de Hodge II, Inst.
  Hautes \'{E}tudes Sci. Publ. Math. No. 40 , 5-57 (1971).
\bibitem{donagimarkaman} R. Donagi, E. Markman. Spectral covers, algebraically completely integrable, Hamiltonian systems, and moduli of bundles, in {\it  Integrable systems and quantum groups }(Montecatini Terme, 1993), 1-119, Lecture Notes in Math., 1620, Springer, Berlin, 1996.
    \bibitem{frila} R. Friedman, R. Laza. Semialgebraic horizontal subvarieties of Calabi-Yau type. Duke Math. J. 162 (2013), no. 12, 2077-2148.
    \bibitem{fulie} L.  Fu.
Beauville-Voisin conjecture for generalized Kummer varieties,  Int. Math. Res. Not. IMRN 2015, no. 12, 3878-3898.
  \bibitem{geemenvoisin} B. van Geemen, C. Voisin. On a conjecture of Matsushita,  Int. Math. Res. Not. IMRN 2016, no. 10, 3111-3123.
\bibitem{grifharris} Ph. Griffiths, J. Harris. Algebraic geometry and local differential geometry. Ann. Sci. \'{E}cole Norm. Sup. (4) 12 (1979), no. 3, 355-452.
    \bibitem{syz} M. Gross. Mirror symmetry and the Strominger-Yau-Zaslow conjecture. Current developments in mathematics 2012, 133-191, Int. Press, Somerville, MA, (2013).
            \bibitem{huybrechts} D. Huybrechts. Curves and cycles on K3 surfaces. Algebr. Geom. 1 (2014), no. 1, 69-106.
\bibitem{hwang} J.-M.  Hwang. Base manifolds for fibrations of projective irreducible symplectic manifolds. Invent. math. 174, 625-644 (2008).
    \bibitem{linthesis}  H.-Y. Lin. PhD. Thesis.
  \bibitem{lin} H.-Y. Lin.  Lagrangian constant cycle subvarieties in Lagrangian fibrations,  arXiv:1510.01437.
  \bibitem{lossen} Ch. Lossen. When does the Hessian determinant vanish identically ? Bull. Braz. math. Soc. 35, 71-32 (2004).
  \bibitem{manin} Yu. Manin. Rational points of algebraic curves over function fields, in {\it Fifteen papers on algebra}, Transl. Amer. Math. Soc. (2) 50, Amer. Math. Soc., Providence, RI, 1966,  pp. 189-234.
\bibitem{matsumain} D. Matsushita. On fibre space structures of a projective irreducible symplectic manifold,
Topology 38 (1999), no. 1, 79-83.

    \bibitem{matsu2} D. Matsushita. On base manifolds of Lagrangian fibrations. Sci. China Math. 58 (2015), no. 3, 531-542.
        \bibitem{matsurank} D.   Matsushita. On deformations of Lagrangian fibrations, in {\it  K3 surfaces and their moduli},  Progr. Math. 315, Birkh\"auser/Springer, (2016) 237–243.
        \bibitem{matsushita13} D. Matsushita. On isotropic divisors on irreducible symplectic manifolds,  arXiv:1310.0896.
            \bibitem{muellerstach}  S. M\"uller-Stach. Constructing indecomposable motivic cohomology classes on algebraic surfaces. J. Algebraic Geom. 6 (1997), no. 3, 513-543.
        \bibitem{mumford}   D. Mumford. Rational equivalence of $0$-cycles on surfaces.
J. Math. Kyoto Univ.  9  (1968) 195-204.
        \bibitem{refMA}  J. Rauch, B.  Taylor. The Dirichlet problem for the multidimensional Monge-Amp\`{e}re equation. Rocky Mountain J. Math. 7 (1977), no. 2, 345-364.
   \bibitem{riess1} U. Riess. On the Chow ring of birational irreducible symplectic varieties, Manuscripta Math. 145 (2014), no. 3-4, 473-501.
        \bibitem{riess2} U. Riess.
    On the Beauville conjecture,  Int. Math. Res. Not. IMRN 2016, no. 20, 6133-6150.
            \bibitem{UY}  E. Ullmo, A. Yafaev. A characterization of special subvarieties. Mathematika 57 (2011), no. 2, 263-273.
                \bibitem{voisinicm} C. Voisin.  Variations of Hodge structure and algebraic cycles. Proceedings of the International Congress of Mathematicians, Vol. 1, 2 (Z\"urich, 1994), 706-715, Birkh\"auser, Basel, (1995).
                \bibitem{voisindensity} C. Voisin. Densit\'{e} du lieu de Noether-Lefschetz pour les sections hyperplanes des vari\'{e}t\'{e}s de Calabi-Yau de dimension 3.  Internat. J. Math. 3 (1992), no. 5, 699-715.
\bibitem{voisinpamq} C. Voisin. On the Chow ring of certain algebraic hyper-K\"ahler manifolds, Pure and Applied Mathematics Quarterly,
 Volume 4, Number 3, (2008).
 \bibitem{voisinlag} C. Voisin. Sur la stabilit\'e des sous-vari\'et\'es lagrangiennes des vari\'et\'es symplectiques holomorphes, in {\it Complex projective geometry (Trieste, 1989/Bergen, 1989)}, 294-303,
London Math. Soc. Lecture Note Ser., 179, Cambridge Univ. Press, Cambridge, 1992.
\bibitem{voisinbook} C. Voisin.   {\it Hodge theory and complex algebraic geometry} II.  Cambridge Studies in Advanced Mathematics, 77. Cambridge University Press, Cambridge, (2003).
\bibitem{zannierbook} U. Zannier. {\it Some problems of unlikely intersections in arithmetic and geometry}. With appendixes by David Masser. Annals of Mathematics Studies, 181. Princeton University Press, Princeton, NJ, (2012).
\end{thebibliography}
    \end{document}